\documentclass[11pt, a4paper]{article}
\usepackage{graphicx}
\usepackage[ruled,vlined]{algorithm2e}
\usepackage{algpseudocode}
\usepackage{algorithmicx}
\usepackage{csquotes}
\usepackage{graphicx}    
\usepackage{caption}     
\usepackage{subcaption} 
\usepackage{amsmath,amsthm,amssymb,enumerate}
\usepackage{color}
\usepackage[utf8]{inputenc}
\catcode`\@=11 \@addtoreset{equation}{section}

\catcode`\@=12
\usepackage{hyperref}
\usepackage{authblk}
\usepackage{etoolbox}
\usepackage{lmodern}
\usepackage[utf8]{inputenc}
\usepackage{amsmath}
\usepackage{amssymb}
\usepackage{amsfonts}
\usepackage{amsthm}
\usepackage{algpseudocode}
\usepackage{float}   

\newtheorem{lemma}{Lemma}
\makeatletter
\patchcmd{\@maketitle}{\LARGE \@title}{\fontsize{15}{5}\selectfont\@title}{}{}
\makeatother

\newtheorem{Theorem}{Theorem}[section]
\newtheorem{Proposition}{Proposition}[section]
\newtheorem{Lemma}{Lemma}[section]
\newtheorem{Corollary}{Corollary}[section]
\newtheorem{Remark}{Remark}[section]
\newtheorem{Definition}{Definition}[section]

\newcommand{\bTheorem}[1]{
	\begin{Theorem} \label{T#1} }
	\newcommand{\eT}{\end{Theorem}}

\newcommand{\bProposition}[1]{
	\begin{Proposition} \label{P#1}}
	\newcommand{\eP}{\end{Proposition}}

\newcommand{\bLemma}[1]{
	\begin{Lemma} \label{L#1} }
	\newcommand{\eL}{\end{Lemma}}

\newcommand{\bCorollary}[1]{
	\begin{Corollary} \label{C#1} }
	\newcommand{\eC}{\end{Corollary}}

\newcommand{\bRemark}[1]{
	\begin{Remark} \label{R#1} }
	\newcommand{\eR}{\end{Remark}}

\newcommand{\bDefinition}[1]{
	\begin{Definition} \label{D#1} }
	\newcommand{\eD}{\end{Definition}}

\newcommand{\bFormula}[1]{ \begin{equation} \label{#1} }
	\newcommand{\eF}{ \end{equation} }

\newcommand{\vu}{\vc{u}}
\newcommand{\vc}[1]{{\bf #1}}
\newcommand{\vX}{\mathrm{X}}

\newcommand{\vU}{\vc{U}}

\newcommand{\vv}{\vc{v}}

\newcommand{\vw}{\mathbf{w}}
\newcommand{\vM}{\mathrm{M}}
\newcommand{\vB}{\mathrm{B}}
\newcommand{\vN}{\mathrm{N}}

\newcommand{\vD}{\mathrm{D}}

\newcommand{\vsigma}{\mathbf{\sigma}}

\newcommand{\vA}{\mathrm{A}}
\definecolor{Cgrey}{rgb}{0.85,0.85,0.85}
\definecolor{Cblue}{rgb}{0.50,0.85,0.85}
\definecolor{Cred}{rgb}{1,0,0}
\definecolor{fancy}{rgb}{0.10,0.85,0.10}

\newcommand\Cbox[2]{%
	\newbox\contentbox%
	\newbox\bkgdbox%
	\setbox\contentbox\hbox to \hsize{%
		\vtop{
			\kern\columnsep
			\hbox to \hsize{%
				\kern\columnsep%
				\advance\hsize by -2\columnsep%
				\setlength{\textwidth}{\hsize}%
				\vbox{
					\parskip=\baselineskip
					\parindent=0bp
					#2
				}%
				\kern\columnsep%
			}%
			\kern\columnsep%
		}%
	}%
	\setbox\bkgdbox\vbox{
		\color{#1}
		\hrule width  \wd\contentbox %
		height \ht\contentbox %
		depth  \dp\contentbox
		\color{black}
	}%
	\wd\bkgdbox=0bp%
	\vbox{\hbox to \hsize{\box\bkgdbox\box\contentbox}}%
	\vskip\baselineskip%
}

\date{}

\newcommand{\vx}{\mathbf{x}}
\newcommand{\vy}{\mathbf{y}}
\newcommand{\vp}{\mathbf{p}}
\newcommand{\vz}{\mathbf{z}}
\newcommand{\vZ}{\mathbf{Z}}

\begin{document}
	{\linespread{2.8} 
		\title{ A numerical method to simulate the stochastic linear-quadratic optimal control problem with control constraint in higher dimensions}}
	\author[1]{Abhishek Chaudhary\thanks{Email: \href{mailto:chaudhary@na.uni-tuebingen.de}{chaudhary@na.uni-tuebingen.de}}}
	\affil[1]{Mathematisches Institut Universität Tübingen, Auf der Morgenstelle 10, 72076 Tübingen, Germany}
	\maketitle
	
	\begin{abstract}
		We propose an {\em implementable} numerical scheme for the discretization of linear-quadratic optimal control problems involving SDEs in higher dimensions with {\em control constraint}. For time discretization, we employ the implicit Euler scheme, deriving discrete optimality conditions that involve time discretization of a backward stochastic differential equations. We develop a recursive formula to compute conditional expectations in the time discretization of the BSDE whose computation otherwise is the most computationally demanding step. Additionally, we present the error analysis for the rate of convergence. We provide numerical examples to demonstrate the efficiency of our scheme in higher dimensions.
		
	\end{abstract}
	
	\noindent
	
	
	\noindent
	
	
	\section{Introduction}
	
	Optimal control problems are central in many fields such as engineering, economics, and finance, where the objective is to find a control policy that minimizes (or maximizes) a certain cost functional over a given time horizon. In real-world scenarios, systems are often subject to uncertainties that can be modeled using stochastic processes. A powerful mathematical tool to describe such systems is the stochastic differential equation (SDE), which allows us to account for randomness and noise in the system dynamics.
	
	In this work, we focus on a class of optimal control problems where the system is governed by a SDE with additive noise. Let $(\Omega,\mathcal{F}, \mathbb{F}=(\mathbb{F}_t)_{t\in[0,T]}, \mathbb{P})$ be a stochastic basis. The objective is to find a $\mathbb{F}$-adapted control process  $\vu^*\equiv\{\vu^*(t)=(\vu_i^*(t))_{1\le\,i\le\,m}:a_i\le\,\vu_i^*(t)\le\,b_i\,\forall\,t\in[0,T]\}$ that minimizes the expected quadratic cost functional $(\alpha>0)$
	\begin{align}\label{cost functional minimum}
		\mathcal{J}(\vx,\vu) =  \frac{1}{2}\mathbb{E} \left[ \int_0^T \left( \left<\vx(s), \vB\vx(s)\right> +\alpha \left\langle \vu(s),  \vu(s) \right\rangle \right) {\rm d}s + \left <\vx(T), \vD\vx(T) \right> \right],
	\end{align}
	subject to the system dynamics described by the following SDE with additive noise
	\begin{align}\label{forward SDE}
		{\rm d}\vx(s)& = \big[\vM \vx(s) + \vN \vu(s)\big]{\rm d}s + \vsigma(s) {\rm d}\vw(s) \quad \forall s \in [0, T],\notag\\
		\vx(0)&=\vx_0.
	\end{align}
	where \( \vx:\Omega\times[0,T]\to\mathbb{R}^d \) is the state variable, \( \vu:\Omega \times[0,T]\to\Pi_{i=1}^m[a_i,b_i]\) is the {\em control constraint}, $\vx_0\in\mathbb{R}^d$ is the initial datum, \( \vsigma:\Omega\times[0,T]\to\mathbb{R}^{d \times k} \) is the noise intensity matrix, and \( \vw \) is a standard \( \mathbb{R}^k \)-valued Brownian motion. For the full list of assumptions, we refer to Section \ref{assumptions}. We refer $\text{\enquote{minimize \eqref{cost functional minimum} subject to \eqref{forward SDE}}}$ as our $\mathbf{SLQ}$ problem.
	\subsection{Previous works}
	In general, optimal control for $\mathbf{SLQ}$ lacks an explicit form, necessitating the use of numerical schemes to approximate solutions. To address this, a Partial Differential Equation (PDE) known as the Hamilton–Jacobi–Bellman (HJB) equation, which is satisfied by the value function, can be derived from the dynamic programming principle (see, e.g., \cite{YongZhou1999}). A Markov chain approximation method is one such approach, as outlined by \cite{Kushner2001} and subsequent works. It is important to note that the derivations for these equations are typically formal, and HJB equations themselves often lack rigorous well-posedness, which limits our understanding of solution properties.
	
	Alternatively, the Pontryagin-type maximum principle offers another approach by transforming the $\mathbf{SLQ}$ problem’s solvability into that of forward-backward stochastic differential equations, facilitating numerical computation of optimal controls. Various numerical schemes leveraging the Pontryagin-type maximum principle have been developed, such as those by \cite{Du2013}, \cite{Dunst2016}, \cite{Gong2017}, \cite{Lu2020}, \cite{ProhlWang2021a}, \cite{ProhlWang2021b}, and \cite{Li2021}, among others.  In \cite{Wang2021}, a time-implicit discretization for stochastic linear quadratic problems subject to SDE with control-dependence noises is proposed, and the convergence rate of this discretization is proved. However, these works often assume deterministic controls or exclude {\em control constraint}.
	
	When {\em control constraints} are present in $\mathbf{SLQ}$ problems, Riccati equation-based methods are not applicable any more. This paper aims to bridge this gap by proposing an approximation method specifically for $\mathbf{SLQ}$ problems involving {\em control constraints}, and we establish rates of convergence for the proposed method.

	\subsection{Our aim and contribution in this work}
	For the linear-quadratic control problem considered here, the Pontryagin's maximum principle (PMP) implies that the optimal control \( \vu^*\) can be characterized using the adjoint process \( \vp \), which satisfies the following backward stochastic differential equation (BSDE):
	\begin{align}\label{adjoint equation}
		\begin{cases}
			{\rm d}\vp(s) = -\big[\vM\vp(s)-\vB\vx^*(s)\big]{\rm d}s + \vz(s){\rm d}\vw(s) & \forall\,s\in[0,T],\\
			\vp(T) = -\vD\vx^*(T),
		\end{cases}
	\end{align}
	where $ \vz$ is a process that represents the martingale component of the adjoint process. The optimal control \( \vu^* \) is then given by
	\begin{align}\label{optimality condition_an}
		\vu^*(s) = \mathcal{P}_{ad} [\frac{1}{\alpha}\vN^{T}\vp(s)]\quad\forall s\in[0,T].
	\end{align}
	Substituting this expression for \( \vu^* \) into the forward SDE and the adjoint BSDE \eqref{adjoint equation}, one obtains a coupled system of forward-backward SDEs (FBSDEs). The solvability of this FBSDE system is a crucial step in characterizing the optimal control and solving the $\mathbf{SLQ}$ problem.
	\subsubsection{Our contribution}
	The objective of this paper is to develop an efficient numerical method for solving the $\mathbf{SLQ}$ problem with the help of the related FBSDEs \eqref{forward SDE}-\eqref{optimality condition_an}. In particular, it is well-known that to numerically solve the (high-dimensional) BSDE \eqref{backward SDE} requires significant computational resources which practically limit this approval to (small) dimensional $d$, see \cite{Dunst2016, Gobet2005}.
	
	The main novelty of our approach lies in the use of a recursive formula (see below \eqref{recursive formula}) to compute the approximate solution to the BSDE \eqref{backward SDE}, which significantly enhances the efficiency of the approximation. This recursive method allows us to iteratively compute the state--control approximation (iterate) even in higher dimensions $d$. We remark that more approximation methods based on regression estimators (statistical learning tools) are computationally expensive and suffer from the curse of dimensionality, see \cite{Dunst2016, Gobet2005}. To the author's best knowledge, this is the first attempt for numerical approximation of stochastic linear quadratic optimal control problems with {{\em control constraints}}.
	
	The contributions of this work can be summarized as follows:
	\begin{itemize}
		
		\item In analytic setup, first we discuss the existence and uniqueness of optimal pair $(\vx^*,\vu^*)$ to $\mathbf{SLQ}$ problem \eqref{cost functional minimum}-\eqref{forward SDE}. We then give the proof of the Pontryagin's maximum principle (Theorem \ref{Lemma_optimality condition}) for the characterization of the optimal pair $(\vx^*, \vu^*)$. 
		
		\item In the second part, we discretize the optimality system \eqref{forward SDE}-\eqref{optimality condition_an} in time using the implicit Euler scheme for SDE \eqref{forward SDE} and BSDE \eqref{adjoint equation}, and then derive the corresponding discrete optimality conditions \eqref{discrete optimality condition}. We already emphasized that the presence of conditional expectations in the discrete optimality system \eqref{characterization of the optimal pair} poses a challenge for direct implementation. To solve this decoupled system \eqref{characterization of the optimal pair}, we apply a projected gradient descent method {\em i.e., Algorithm \ref{tt}}, which requires again the computation of conditional expectations $\mathbb{E}[\mathbf{\cdot} |\mathcal{F}_{t_n}]$ to obtain the gradient of the cost function related to discrete $\mathbf{SLQ}_{h}$ problem.
		
		\item To avoid costly simulations, we develop a recursive formula \eqref{recursive formula} for the exact computation of adjoint variables $\vp_{h}^{(\ell)}$ in the decoupled system \eqref{characterization of the optimal pair}. This formula leverages the iterative nature of the projected gradient method, the linearity of the optimality system and the additive nature of noise, enabling efficient approximation of the optimal solution $(\vx^*,\vu^*)$ to the $\mathbf{SLQ}$ problem in high-dimensional $d$, see Proposition \ref{Proposition recursive formula}, {\em  Algorithm \ref{Algorithm_implementable}} and numerical example \ref{example}.
		
		\item In the end, we provide a comprehensive analysis that leads to the convergence rates; see Theorem \ref{theorem 4.1}. The error analysis also encompasses a wide range of $\mathbf{SLQ}$ problems related to SDE, which result from finite-dimensional approximations of $\mathbf{SLQ}$ problems governed by (infinite-dimensional) stochastic partial differential equations (SPDEs). Specifically, the matrix $\vM$ can be viewed as representing a finite-dimensional approximation of an unbounded operator acting on an infinite-dimensional Sobolev space, see Remark \ref{Remark 4.0}. We also give the error plots for our method for the different cases, see numerical example \ref{numerical example 01}.
		
	\end{itemize}
	\noindent
	The remainder of this paper is organized as follows. In Section~\ref{section 2}, we describe the notations, list of assumptions, and analytic results. Section~\ref{section 3} presents the details of the time discretization of the $\mathbf{SLQ}$ problem, the projected gradient descent method, the recursive computation for the adjoint variables, the {\em implementable} algorithm, and a numerical example. Section~\ref{section 4} contains the convergence analysis of the {\em implementable} algorithm and error plots.
	
	\noindent
	
	\section{Notations, assumptions and analytic results}\label{section 2}
	In this section, we present some notations, a list of assumptions and the preliminary results necessary for our analysis.
	\subsection{Notations}	
	\noindent
	\textbf{Functional spaces:} Let \( H \) be a separable Hilbert space. We define \( L^2_{\mathbb{F}}(\Omega \times [0,T]; H) \) as the space of all \( H \)-valued predictable and \( L^2 \)-integrable functions on \( \Omega \times [0,T] \). For brevity, we denote the following spaces:
	\begin{align*}
		\vX &= L^2(\Omega; C([0,T]; \mathbb{R}^d)) \cap L^2_{\mathbb{F}}(\Omega \times [0,T]; \mathbb{R}^d),\qquad
		\vU=L^2_{\mathbb{F}}(\Omega \times [0,T]; \mathbb{R}^m)),\\& \vZ=L^2_{\mathbb{F}}(\Omega \times [0,T]; \mathbb{R}^{d\times k}),\qquad \vU_{ad} =\{\vu\in\vU: \vu\in\Pi_{i=1}^m [a_i, b_i])\}.
	\end{align*}
	Let $\|\cdot\|_{C_t\mathbb{L}_\omega^2}$ denote the following norm on $C([0,T];L^2(\Omega;\mathbb{R}^{d\times  k}))$: for all $\Phi \in C([0,T];L^2(\Omega;\mathbb{R}^{d\times  k}))$,
	\begin{align*}
		\|\Phi\|_{C_t\mathbb{L}^2_\omega}^2=\sup_{t\in[0,T]}\mathbb{E}\big[\|\Phi(t)\|^2\big].
	\end{align*}
	For any $l\in\mathbb{N}$, we denote the \(\mathbb{L}^2\) norm as follows: for all \(\mathbf{x} \in L^2(\Omega \times [0, T]; \mathbb{R}^l)\),
	
	\[
	\|\mathbf{x}\|_{\mathbb{L}^2}^2 = \mathbb{E} \left[ \int_0^T \|\mathbf{x}(s)\|^2 \, \mathrm{d}s \right].
	\]
	Let $\left<\left<\cdot,\cdot\right>\right>$ denote the inner product on $\vU$, while $\left<\cdot,\cdot\right>$ denotes the inner product on $\mathbb{R}^l$ for any $l\in \mathbb{N}$.

	\noindent
	\textbf{Matrix norms:} Let \( \vA = [a_{ij}] \in \mathbb{R}^{l_1 \times l_2} \) be a matrix of order \( l_1 \times l_2 \). The Frobenius norm of matrix \( \vA \) is defined as:
	\begin{align*}
		\|\vA\| = \sqrt{\sum_{i=1}^{l_1} \sum_{j=1}^{l_2} a_{ij}^2}.
	\end{align*}
	The sup-norm of \(\vA \) is denoted by \( \|\vA\|_{\infty} \) and is defined as:
	\begin{align*}
		\|\vA\|_{\infty} = \sup_{1 \leq i \leq l_1, 1 \leq j \leq l_2} |a_{ij}|.
	\end{align*}
	\textbf{Time-discretization:} Consider the time interval $[0, T]$, which we uniformly discretize into $N$ time steps. The step size is denoted by $h = \frac{T}{N}$, and the discrete times are given by $t_n =nh$ for $n = 0, 1, \ldots, N$. 
	At these discrete time points, we approximate the state and adjoint variables. To facilitate our analysis, we define the following spaces:
	\begin{align*}
		\vU_{h} &= \left\{\vu_{h}\in \vU: \vu_{h}(s) = \vu_n \text{ for } s \in [t_n,t_{n+1}), n = 0, \ldots, N-1 \right\}, \\
		\vX_{h} &= \left\{\vx_{h}\in \vX : \vx_{h}(s) = \vx_n \text{ for } s \in [t_n,t_{n+1}), n = 0, \ldots, N \right\},\\
		\vU_{ad}^h&=\vU_h\cap \vU_{ad}.
	\end{align*}
	These spaces represent the sets of piecewise constant functions on the time interval $[0, T]$, where the control and state variables take constant values on each sub-interval $[t_n, t_{n+1})$. The norms associated with these spaces are defined as follows:
	\begin{align*}
		\|\vu_{h}\|_{\vU_{h}} = \left( h\sum_{n=0}^{N-1}\|\vu_n\|^2 \right)^{1/2},\qquad
		\|\vx_{h}\|_{\vX_{h}} = \left( h\sum_{n=1}^{N}\|\vx_n\|^2 \right)^{1/2}.
	\end{align*}
	These norms measure the $L^2$ norm of the discrete control and state variables over the time interval $[0, T]$. We introduce a projection operator $\Pi_h: \vU_{ad} \to \vU_{ad}^h$ defined as follows: for all $\vu \in \vU_{ad}$,
	\begin{align*}
		[\Pi_h \vu](\cdot,s) = \vu(\cdot,t_n) \quad \forall\, s \in [t_n,t_{n+1}), \quad n = 0,\ldots,N-1.
	\end{align*}
	This projection effectively maps the continuous control function $\vu_{h}$ onto a piecewise constant function over each time interval $[t_n, t_{n+1})$.	
	\subsection{Assumptions}\label{assumptions}
	In this article, we assume the following conditions on the given data:
	
	\begin{itemize}
		\item[1.] The matrices \(\vB\) and \(\vD\) are positive semi-definite of order \(d \times d\).
		\item[2.] The matrices \(\vM\) and \(\vN\) are of order \(d \times d\) and \(d \times m\), respectively. Additionally, \(\vM\) can be expressed as \(\vM = -\vM_1\vM_1^{T}\) for some matrix \(\vM_1\) of order \(d \times d\).
		\item[3.] The noise intensity matrix \(\sigma\) satisfies \(\vsigma \in L^2_{\mathbb{F}}(\Omega\times[0,T]; \mathbb{R}^{d \times k})\cap C([0,T];L^2(\Omega;\mathbb{R}^{d\times k}))\).
		\item[4.] The process \(\vw \equiv \{\vw(s); s \in [0,T]\}\) is a \(\mathbb{F}\)-adapted, \(\mathbb{R}^k\)-valued Wiener process.
		\item[5.] The values \(a_i, b_i \in \mathbb{R}\) for all \(i = 1, \ldots, m\), and \(m, n, d \in \mathbb{N}\).
		\item[6.] The orthogonal projection \(\mathcal{P}_{ad}:\mathbb{R}^m \to \Pi_{i=1}^m [a_i, b_i]\) is defined by
		\[
		[\mathcal{P}_{ad} \vu]_i = \min\left(\max(a_i, \vu_i), b_i\right), \quad \forall\, 1 \leq i \leq m, \, \vu \in \mathbb{R}^m.
		\]
	\end{itemize}
	
	\begin{Remark}[dependence of constants]
		Throughout the latter part of this manuscript, we introduce a constant \(C_{T,d}\) which depends only on the final time \(T\), $\alpha$ and the norms of the matrices \( \{ \vB, \vD, \vN\} \). Importantly, this constant \(C_{T,d}\) is independent of the norms of \(\vM\) and \(\vM_1\). Moreover, we use \(C_T\) to denote a constant that depends solely on the time horizon \(T\), while \(c_\delta\) denotes a constant that depends only on the parameter \(\delta\). While constants in the estimates may vary from line to line, such changes do not impact the overall results.
	\end{Remark}

	\subsection{Analytical results}
	In this subsection, we present analytic results concerning the existence and uniqueness of solutions to the forward stochastic differential equation (SDE) \eqref{forward SDE}, the backward SDE \eqref{adjoint equation}, and the stochastic linear-quadratic (SLQ) problem \eqref{cost functional minimum}-\eqref{forward SDE}. These results form the foundation of the subsequent analysis and are crucial for ensuring the well-posedness of $\mathbf{SLQ}$ problem and the error analysis.
	
	\begin{Lemma}[forward SDE]
		\label{Lemma 2.1}
		Let $\vu\in\vU$ and $\vx_0\in\mathbb{R}^d$. Then, there exists a unique solution $\vx\in\vX$ to  SDE \eqref{forward SDE} such that there exists a positive constant $C_T$ such that
		\begin{align}\label{stability for forward SDE}
			\sup_{s\in[0,T]}\mathbb{E}\left[\|\vx(s)\|^2\right] +\|\vM_1 \vx\|^2_{\mathbb{L}^2}\leq C_T\left(\|\vx_0\|^2 + \|\vN\vu(\cdot)\|_{\mathbb{L}^2}^2 + \|\vsigma\|_{\mathbb{L}^2}^2\right),
		\end{align}
		\begin{align}\label{today02}
			\sup_{s\in[0,T]}\mathbb{E}\big[\|\vM_1\vx(s)\|^2\big]+\|\vM\vx\|^2_{\mathbb{L}^2}\le\,C_T\big(\|\vM_1\vx_0\|^2+\|\vN\vu\|_{\mathbb{L}^2}^2+ \|\vM_1 \sigma\|_{\mathbb{L}^2}^2\big),
		\end{align}
		and
		\begin{align}\label{gradient stability for forward SDE}
			\sup_{s\in[0,T]}\mathbb{E}\big[\|\vM\vx(s)\|^2\big]+\|\vM_1\vM \vx\|^2_{\mathbb{L}^2}\le\,C_T(\|\vM\vx_0\|^2+ \|\vM_1\vN\vu\|_{\mathbb{L}^2}^2 + \|\vM\sigma\|_{\mathbb{L}^2}^2).
		\end{align}
		\begin{proof} Existence and uniqueness follow from the standard SDE theory, see \cite{Baldi}. The proof of estimates is a direct consequence of It\^o formula. For completeness of the proof, we give the details. For estimate \eqref{stability for forward SDE}, we apply It\^o's formula to the function $\vx\to \|\vx\|^2$, to get $\mathbb{P}$-a.s., for all $s\in[0,T]$,
			\begin{align*}
				\left<\vx(s), \vx(s)\right>=&\|\vx_0\|^2 + 2\int_{0}^{s}\left<\vx(r), \vM\vx(r)\right>{\rm d}r+ 2\int_0^s\left<\vx(r),\vN\vu(r)\right>{\rm d}r+ 2\int_0^s\left<\vx(r),\sigma(r)\right>{\rm d}\vw(r)\\&\qquad+\int_0^{s}\mbox{Tr}\big(\sigma(r)\sigma^{T}(r)\big){\rm d}r.
			\end{align*}
			It implies that 
			\begin{align*}
				\mathbb{E}\big[\|\vx(s)\|^2\big]+2\mathbb{E}\bigg[\int_0^s\|\vM_1\vx(r)\|^2{\rm d}s\bigg]\le&\|\vx_0\|^2+ \mathbb{E}\big[\int_0^s\|\vx(r)\|^2{\rm d}r\big] + \mathbb{E}\bigg[\int_0^s \|\vN \vu(r)\|^2{\rm d}r\bigg]\\&\qquad+ \mathbb{E}\big[\int_0^s \|\sigma(r)\|^2{\rm d}r\big].
			\end{align*}
			By using Gr\"onswall's inequality, we have
			\begin{align*}
				\sup_{s\in[0,T]}\mathbb{E}\big[\|\vx(s)\|^2\big]+2\mathbb{E}\bigg[\int_0^T\|\vM_1\vx(r)\|^2{\rm d}s\bigg]\le&\,\,e^{T}\mathbb{E}\bigg[\|\vx_0\|^2+\int_0^T \|\vN \vu(r)\|^2{\rm d}r + \int_0^T\|\sigma(r)\|^2{\rm d}r\bigg].
			\end{align*}
			Similarly, we can apply It\^o to $\vx\to\|\vM_1\vx\|^2$ and $\vx\to\|\vM\vx\|^2$ to obtain the estimate \eqref{today02} and \eqref{gradient stability for forward SDE}, respectively.
		\end{proof}
	\end{Lemma}
	
	Next, we present a result related to the backward SDE \eqref{adjoint equation}. 
	\begin{Lemma}[backward SDE]
		Let $\vx\in\vX$. Then, there exists a unique solution $(\vp, \vz)\in\vX\times\vZ$ to the following BSDE
		\begin{align}\label{backward SDE}
			\begin{cases}
				{\rm d}\vp(s) = -\left[\vM\vp(s) -\vB\vx(s)\right]{\rm d}s + \vz(s){\rm d}\vw(s), & \forall\, s \in [0,T], \\
				\vp(T) =-\vD\vx(T).
			\end{cases}
		\end{align}
		Moreover, there exists a positive constant $C_T$ such that
		\begin{align}\label{stability for backward SDE}
			\sup_{s\in[0,T]}\mathbb{E}\left[\|\vp(s)\|^2\right]+\|\vM_1\vp\|^2_{\mathbb{L}^2} + \|\vz(\cdot)\|_{\mathbb{L}^2}^2 \leq C_T\big(\|\vB \vx(\cdot)\|_{\mathbb{L}^2}^2+\mathbb{E}\big[\|\vD\vx(T)\|^2\big]\big),
		\end{align}
		and 
		\begin{align}\label{gradient estimate for BSDE}
			\sup_{s\in[0,T]}\mathbb{E}\big[\|\vM_1\vp(s)\|^2\big]+\|\vM\vp\|^2_{\mathbb{L}^2}+\|\vM_1\vz\|^2_{\mathbb{L}^2} \le\,C_T\big(\mathbb{E}\big[\|\vM_1\vD\vx(T)\|^2\big]+\|\vB\vx\|_{\mathbb{L}^2}^2\big).
		\end{align}
		\begin{proof}
			This existence and uniqueness follow from the standard BSDE theory, see \cite{YongZhou1999, LuZhang2021}. For estimates, one can follow similar lines as done in the proof of Lemma \ref{Lemma 2.1}.
		\end{proof}
	\end{Lemma}
	\begin{Corollary}[Time-regularity]\label{Corrolary 2.1}
		Let $\vx\in \vX$. Let $\vp[\vx] \in \vX$ be the unique solution to the BSDE \eqref{backward SDE}. Then there exists a positive constant $C_T$ such that
		\begin{align}\label{today01}
			\|\vp[\vx] - \Pi_h \vp[\vx]\|_{\mathbb{L}^2}^2 \le C_{T} h (\mathbb{E}\big[\| \vM_1\vD\vx(T))\|^2\big] + \|\vB\vx\|_{\mathbb{L}^2}^2).
		\end{align}
	\end{Corollary}
	\begin{proof}For convenience we take $\vp[\vx]=\vp$. From \eqref{backward SDE}, we obtain for $t\in[t_n, t_{n+1}]$, $\mathbb{P}$-almost surely,
		\begin{align*}
			\|\vp(t)-\vp(t_n)\|^2\le C\bigg(\int_{t_{n}}^{t}\|\vM\vp(s)\|^2{\rm d}s + \int_{t_{n}}^t\|\vB\vx(s)\|^2{\rm d}s +\|\int_{t_n}^{t}\vz(s){\rm d}\vw(s)\|^2\bigg).
		\end{align*}
		With the help of It\^o isometry, we can conclude that
		\begin{align*}
			\mathbb{E}\bigg[\int_{{t_n}}^{t_{n+1}}\|\vp(t)-\vp(t_n)\|^2{\rm d}t\bigg]\le\,C\,h\mathbb{E}\bigg[\int_{t_n}^{t_{n+1}}\|\vM\vp(s)\|^2{\rm d}s +\int_{t_n}^{t_{n+1}}\|\vB\vx(s)\|^2{\rm d}s +\int_{t_{n}}^{t_{n+1}}\|\vz(s)\|^2{\rm d}s\bigg].
		\end{align*}
		It implies that
		\begin{align*}
			\|\vp-\Pi_h\vp\|_{\mathbb{L}^2}^2\le C\, h\bigg(\|\vM\vp\|_{\mathbb{L}^2}^2+\|\vB\vx\|^2_{\mathbb{L}^2} +\|\vz\|_{\mathbb{L}^2}^2\bigg).
		\end{align*}
		Now we use \eqref{stability for backward SDE}-\eqref{gradient estimate for BSDE} to conclude the result.
	\end{proof}
	To derive the optimality conditions and perform error analysis, we first introduce some notations and estimates for the random differential equation. For any \(\vu \in \vU\), let $\vx[\vu]$ denote the unique solution to \eqref{forward SDE} and let \(\vy \equiv \vy[\vu] \in \vX\) denote the unique solution to the following random differential equation:
	\begin{align}\label{random differential equation}
		\begin{cases}
			\vy_t(t) = \vM \vy(t) + \vN \vu(t), & \forall t \in (0,T], \\
			\vy(0) = 0.
		\end{cases}
	\end{align}
	Let \(\vx^\sigma \in \vX\) be the unique solution to the following stochastic differential equation:
	\begin{align}\label{sigmasde}
		\begin{cases}
			\mathrm{d}\vx^\sigma(t) = \vM \vx^\sigma(t) \, \mathrm{d}t + \sigma(t) \, \mathrm{d}\vw(t), & \forall t \in (0,T], \\
			\vx^\sigma(0) = \vx_0.
		\end{cases}
	\end{align}
	Note that 
	\begin{align}\label{today00001}
		\vx^\sigma = \vx[0]\qquad \text{and}\qquad \vx[\vu] = \vy[\vu] + \vx^\sigma\qquad\forall\,\vu \in \vU.
	\end{align} 
	\begin{Lemma} For given $\vu\in\vU$, there exists a unique solution $\vy\equiv\vy[\vu]\in\vX$ to \eqref{random differential equation} satisfies the following estimates
		\begin{align}\label{stability for forward SDE_1}
			\sup_{s\in[0,T]}\mathbb{E}\left[\|\vy(s)\|^2\right] +\|\vM_1 \vy\|^2_{\mathbb{L}^2}\leq C_T\|\vN\vu(\cdot)\|_{\mathbb{L}^2}^2,
		\end{align}
		and
		\begin{align}\label{gradient stability for forward SDE_1}
			\sup_{s\in[0,T]}\mathbb{E}\big[\|\vM\vy(s)\|^2\big]+\|\vM_1\vM \vy\|^2_{\mathbb{L}^2}\le\,C_T\|\vM_1\vN\vu\|_{\mathbb{L}^2}^2.
		\end{align}
	\end{Lemma}
	\begin{proof}
		This proof is similar to Lemma \ref{Lemma 2.1}.
	\end{proof}
	
	\begin{lemma}[time regularity estimate]\label{time regularity estimates for forward SPDE} Let $\vu\in\vU$. Let $\vy\equiv\vy[\vu]\in {\vX}$ be the unique solution to \eqref{forward SDE} with initial data $\vy(0)=0$. Then the following time-regularity estimate holds:
		\begin{align}\label{lll}
			\sum_{n=0}^{N}\int_{t_n}^{t_{n+1}}\|\vM_1(\vy(s)-\vy(t_{n+1}))\|^2{\rm d}s\le\,C_T\,h^2\,\|\vM_1\vN\vu\|_{\mathbb{L}^2}^2.
		\end{align}
	\end{lemma}
	\begin{proof}
		We prove this result in the following two steps.
		
		\noindent
		\textbf{Step 1.} For all $t \in [0, T]$, we have
		\[
		\vM_1 \vy_t(t) = \vM_1 \vM \vy(t) + \vM_1 \vN \vu(t).
		\]
		This implies that
		\[
		\|\vM_1 \vy_t\| \le \|\vM_1 \vM \vy\|_{\mathbb{L}^2}^2 + \|\vM_1 \vN \vu\|_{\mathbb{L}^2}^2.
		\]
		Using the estimate from \eqref{gradient stability for forward SDE_1}, we obtain
		\begin{align}\label{today03}
			\|\vM_1 \vy_t\|_{\mathbb{L}^2}^2 \le C_T \|\vM_1 \vN \vu\|_{\mathbb{L}^2}^2.
		\end{align}
		
		\noindent
		\textbf{Step 2.} By applying Hölder's inequality, we conclude that
		\begin{align*}
			\int_{t_n}^{t_{n+1}} \|\vM_1 (\vy(s) - \vy(t_{n+1}))\|^2 \, \mathrm{d}s &= \int_{t_n}^{t_{n+1}} \bigg\|\int_s^{t_{n+1}} \vM_1 \vy_t(r) \, \mathrm{d}r \bigg\|^2 \, \mathrm{d}s \\
			&\le \int_{t_n}^{t_{n+1}} \bigg(\int_s^{t_{n+1}} \|\vM_1 \vy_t(r)\| \, \mathrm{d}r\bigg)^2 \, \mathrm{d}s \\
			&\le \int_{t_n}^{t_{n+1}} (t_{n+1} - s) \int_s^{t_{n+1}} \|\vM_1 \vy_t(r)\|^2 \, \mathrm{d}r \, \mathrm{d}s \\
			&\le h^2 \int_{t_n}^{t_{n+1}} \|\vM_1 \vy(r)\|^2 \, \mathrm{d}r.
		\end{align*}
		This shows that
		\[
		\sum_{n=0}^{N} \int_{t_n}^{t_{n+1}} \|\vM_1 (\vy(s) - \vy(t_{n+1}))\|^2 \, \mathrm{d}s \le h^2 \|\vM_1 \vy_t\|_{\mathbb{L}^2}^2.
		\]
		Finally, using the estimate \eqref{today03}, we obtain the desired result.
	\end{proof}

	The following lemma confirms that for any initial state, there exists a unique optimal pair of state and control processes, ensuring the well-posedness of the $\mathbf{SLQ}$ problem.
	
	\begin{Theorem}[Pontryagin's maximum principle] \label{Lemma_optimality condition}
		For any $\vx_0\in\mathbb{R}^d$, there exists a unique optimal pair $(\vx^*, \vu^*)\in\vX\times\vU_{ad}$ to the $\mathbf{SLQ}$ problem \eqref{cost functional minimum}-\eqref{adjoint equation}. Moreover, there exists a positive constant $C_{T,d}$ such that
		\begin{align}\label{L^2 bound for optimal pair}
			\|\sqrt{\vB}\vx^*(\cdot)\|_{\mathbb{L}^2}^2 + \|\sqrt{\alpha}\vu^*(\cdot)\|_{\mathbb{L}^2}^2 +\mathbb{E}\big[\|\sqrt{\vD}\vx^*(T)\|^2\big] \leq C_{T,d}(\|\vx_0\|^2+\|\sigma\|_{\mathbb{L}^2}^2).
		\end{align}
		Moreover, the following optimal equality holds,
		\begin{align}\label{optimality condition}
			\vu^*(s) =\mathcal{P}_{ad}[\frac{1}{\alpha}\vN^{T}\vp[\vx^*](s)]\quad\forall s\in[0,T],
		\end{align}
		where $\vp[\vx^*]\in\vX$ is first solution component to \eqref{backward SDE}, and integral inequality holds, for every $\vu\in\vU_{ad}$,
		\begin{align}\label{today06}
			\mathbb{E}\bigg[\int_0^T\left<\vy[\vu-\vu^*](s),\vB\vx^*(s)\right>+\alpha\left<\vu(s)-\vu^*,\vu^*(s)\right>{\rm d}s + \left<\vy[\vu-\vu^*](T),\vD\vx^*(T)\right>\bigg]\ge\,0.
		\end{align}
		
		\begin{proof}
			We prove this result through the following steps:
			
			\noindent
			\textbf{Step 1.} Define the reduced optimal control $\mathcal{\hat{J}}:\vU\to \mathbb{R}$ by
			\[
			\mathcal{\hat{J}}(\vu) = \mathcal{J}(\vx[\vu], \vu).
			\]
			It is clear that $\mathcal{\hat{J}}$ is Fréchet differentiable on $\vU$ and strictly convex due to the quadratic nature of the cost functional. Therefore, there exists a global minimum $\vu^* \in \vU_{ad}$ such that
			\[
			\mathcal{\hat{J}}(\vu^*) = \min_{\vu \in \vU_{ad}} \mathcal{\hat{J}}(\vu).
			\]
			Since $\vu^*$ is a global minimizer for $\mathcal{\hat{J}}$, it follows that for all $\vu \in \vU_{ad}$ (see \cite[Theorem 1.46]{Hinze}),
			\begin{align}\label{today34}
				\left<\left\langle \mathcal{\hat{J}}'(\vu^*), \vu - \vu^* \right\rangle\right> \geq 0.
			\end{align}
			
			\noindent
			\textbf{Step 2.} The Fréchet derivative of $\mathcal{\hat{J}}$ at $\vu^*$ in the direction $\vv \in \vU$ is given by
			\[
			\left<\left\langle \mathcal{\hat{J}}'(\vu^*), \vv \right\rangle\right> = \mathbb{E} \bigg[ \int_0^T \big( \left\langle \vy[\vv](s), \vB \vx^*(s) \right\rangle +\alpha \left\langle \vv(s), \vu^*(s) \right\rangle \big) \, \mathrm{d}s + \left\langle \vy[\vv](T), \vD \vx^*(T) \right\rangle \bigg].
			\]
			Using this expression with $\vv = \vu - \vu^*$, we obtain for all $\vu \in \vU_{ad}$
			\begin{align}\label{today35}
				&\left<\left\langle \mathcal{\hat{J}}'(\vu^*), \vu - \vu^* \right\rangle\right> \notag\\&= \mathbb{E} \bigg[ \int_0^T \big( \left\langle \vy[\vu - \vu^*](s), \vB \vx^*(s) \right\rangle + \alpha\left\langle \vu(s) - \vu^*(s), \vu^*(s) \right\rangle \big) \, \mathrm{d}s + \left\langle \vy[\vu - \vu^*](T), \vD \vx^*(T) \right\rangle \bigg].
			\end{align}
			Combining \eqref{today34} and \eqref{today35}, we obtain inequality \eqref{today06}.
			
			\noindent
			\textbf{Step 3.} Let $\vv \in \vU$ be arbitrary. Applying Itô's formula to the pairing $\left\langle \vy[\vv], \vp[\vx^*] \right\rangle$, we get, $\mathbb{P}$-a.s.,
			\begin{align*}
				\left\langle \vy[\vv](T), \vp[\vx^*](T)  \right\rangle-\left\langle \vy[\vv](0), \vp[\vx^*](0)  \right\rangle=&\int_0^T \left\langle \vy_t[\vv](s), \vp[\vx^*](s) \right\rangle \, \mathrm{d}s + \int_0^T \left\langle \vy[\vv](s), \vM \vp[\vx^*](s) \right\rangle \, \mathrm{d}s \\
				&\qquad + \int_0^T \left\langle \vy[\vv](s), \vB \vx^*(s) \right\rangle \, \mathrm{d}s + \int_0^T \left\langle \vy[\vv](s), \vz(s) \right\rangle \, \mathrm{d} \vw(s).
			\end{align*}
			Since $\vy(0)  = 0$ and $\vp[\vx^*](T) = -\vD \vx^*(T)$, we conclude that
			\begin{align}\label{today37}
				\int_0^T \left\langle \vy[\vv](s), \vB \vx^*(s) \right\rangle \, \mathrm{d}s + \left\langle \vy[\vv](T), \vD \vx^*(T) \right\rangle =& -\int_0^T \left\langle \vy_t[\vv](s), \vp[\vx^*](s) \right\rangle \, \mathrm{d}s - \int_0^T \left\langle \vM \vy[\vv](s), \vp[\vx^*](s) \right\rangle \, \mathrm{d}s \notag\\
				&\qquad - \int_0^T \left\langle \vy[\vv](s), \vz(s) \right\rangle \, \mathrm{d} \vw(s).
			\end{align}
			Since $\vy[\vv]$ solves equation \eqref{random differential equation}, we also have
			\begin{align}\label{today38}
				\int_0^T \left\langle \vy_t(s), \vp[\vx^*](s) \right\rangle \, \mathrm{d}s = \int_0^T \left\langle \vM \vy[\vv](s), \vp[\vx^*](s) \right\rangle \, \mathrm{d}s + \int_0^T \left\langle \vN \vv, \vp[\vx^*] \right\rangle \, \mathrm{d}s.
			\end{align}
			Substituting \eqref{today37} and \eqref{today38} and noting that
			\[
			\mathbb{E} \bigg[ \int_0^T \left\langle \vy[\vv], \vp[\vx^*] \right\rangle \, \mathrm{d} \vw(s) \bigg] = 0,
			\]
			we obtain
			\begin{align}\label{today39}
				\mathbb{E} \bigg[ \int_0^T \big( \left\langle \vy[\vv](s), \vB \vx^*(s) \right\rangle +\alpha \left\langle \vv(s), \vu^*(s) \right\rangle \big) \, \mathrm{d}s + \left\langle \vy[\vv](T), \vx^*(T) \right\rangle \bigg] = \mathbb{E} \bigg[ \int_0^T \left\langle \vv, \alpha\vu^* - \vN^{T} \vp[\vx^*] \right\rangle \, \mathrm{d}s \bigg].
			\end{align}
			For all $\vu \in \vU_{ad}$, by using \eqref{today06} and setting $\vv = \vu - \vu^*$ in \eqref{today39}, we conclude that
			\[
			\mathbb{E} \bigg[ \int_0^T \left\langle \vu(s) - \vu^*(s), \vu^* - \frac{1}{\alpha}\vN^{T} \vp[\vx^*] \right\rangle \, \mathrm{d}s \bigg] \geq 0.
			\]
			This implies that
			\[
			\vu^* = \mathcal{P}_{ad}[\frac{1}{\alpha}\vN^{T} \vp[\vx^*]].
			\]
			For the estimate \eqref{L^2 bound for optimal pair}, we can conclude that
			\begin{align*}
				\mathcal{J}(\vx^*,\vu^*)\le \mathcal{J}(\vx[0],\vu[0])\le C_{T,d}(\|\vx_0\|^2+ \|\sigma\|_{\mathbb{L}^2}^2). 
			\end{align*}
		\end{proof}
	\end{Theorem}
	\begin{Remark}\label{important remark}
		In the case of free control, \emph{i.e.}, \(\vU_{ad} = \vU\), the integral identity \eqref{today06} evaluates to zero. By following similar lines, we can prove that for all $\vu\in\vU$,
		\begin{align*}
			\mathcal{\hat{J}}'(\vu)= \alpha \vu -\vN^{T}\vp[\vx[\vu]].
		\end{align*}
	\end{Remark}
	At the stage of the error analysis, establishing the time regularity of the optimal control is essential. To this end, we present the following corollary. In the proof, the optimality condition \eqref{optimality condition} plays a key role in enhancing the time regularity of the optimal control.
	\begin{Corollary}[Time-regularity estimate]\label{Corollary 2.1} Let $\vu^*\in\vU_{ad}$ be a unique optimal control to $\mathbf{SLQ}$ problem. Then $\vu^*\in\vU_{ad}$ satisfies the following time regularity estimate
		\begin{align}\label{today10}
			\|\vu^*-\Pi_h\vu^*\|^2_{\mathbb{L}^2}\le\,C_{T,d}\,h \big(\|\vx_0\|^2+\|\vM_1\vx_0\|^2+ \|\sigma\|_{\mathbb{L}^2}^2+\|\vM_1\sigma\|_{\mathbb{L}^2}^2\big).
		\end{align}
	\end{Corollary}
	\begin{proof}
		Starting from the optimality condition \eqref{optimality condition} and using the stability property of the projection operator $\mathcal{P}_{ad}$, we have
		\begin{align*}
			\|\vu^* - \Pi_h \vu^*\|_{\mathbb{L}^2}^2 &= \|\mathcal{P}_{ad} \bigg[\frac{1}{\alpha}\vN(\vp[\vx^*] - \Pi_h \vp[\vx^*])\bigg]\|_{\mathbb{L}^2}^2 \\
			&\le\frac{1}{\alpha} \|\vN (\vp[\vx^*] - \Pi_h \vp[\vx^*])\|_{\mathbb{L}^2}^2.
		\end{align*}
		Then, by applying inequality \eqref{today01}, we obtain
		\begin{align*}
			\|\vu^* - \Pi_h \vu^*\|_{\mathbb{L}^2}^2 &\le\frac{1}{\alpha} \|\vN\|^2 \|\vp[\vx^*] - \Pi_h \vp[\vx^*]\|_{\mathbb{L}^2}^2 \\
			&\le\frac{1}{\alpha} \|\vN\|^2 \, h C_{T} \big( \|\vM_1 \vD \vx^*(T)\|^2 + \|\vB \vx^*\|_{\mathbb{L}^2}^2 \big).
		\end{align*}
		Using inequality \eqref{today02}, we further obtain
		\begin{align*}
			\|\vu^* - \Pi_h \vu^*\|_{\mathbb{L}^2}^2 &\le C_{T, d} h \big( \|\vM_1 \vx^*(T)\|^2 + \|\vx^*\|_{\mathbb{L}^2}^2 \big) \\
			&\le C_{T, d} h \big( \|\vM_1 \vx_0\|^2 + \|\vN \vu^*\|_{\mathbb{L}^2}^2 + \|\vM_1 \sigma\|_{\mathbb{L}^2}^2 \big).
		\end{align*}
		Finally, applying the $\mathbb{L}^2$-bound for the optimal pair \eqref{L^2 bound for optimal pair}, we conclude that
		\begin{align*}
			\|\vu^* - \Pi_h \vu^*\|_{\mathbb{L}^2}^2 &\le C_{T, d} h \big( \|\vx_0\|^2 + \|\vM_1 \vx_0\|^2 + \|\sigma\|_{\mathbb{L}^2}^2 + \|\vM_1 \sigma\|_{\mathbb{L}^2}^2 \big).
		\end{align*}
	\end{proof}
	
	\begin{Remark}
		Note that by explicitly tracking the constant $C_{T, d}$ in the above proof, we obtain
		\begin{align*}
			C_{T, d} = C_{T} \|\vN\|^2 \max \{ \|\vD\|_{\infty}, \|\sqrt{\vB}\|^2, \|\vN\|^2 \}.
		\end{align*}
	\end{Remark}
	\section{Time discretization and an efficient {implementable} scheme}\label{section 3}
	For a given discrete control $\{\vu_n\}_{n=0}^{N-1} \equiv \vu_{h} \in \vU_{h}$, we approximate the state variable using the implicit Euler scheme. The resulting discrete state variable $\vx_{h} \equiv \{\vx_n\}_{n=0}^{N} \in \vX_{h}$ is governed by the following difference equation:
	\begin{align*}
		\vx_{n+1} &= \vx_n + h\left[\vM\vx_{n+1} + \vN\vu_n\right] + \vsigma(t_n) \Delta\vw_n, \qquad \vx_0 = \vx,
	\end{align*}
	where $\Delta\vw_n = \vw_{n+1} - \vw_n$ represents the increment of the Brownian motion over the time step $[t_n, t_{n+1}]$. Rearranging the above equation to solve for $\vx_{n+1}$, we obtain:
	\begin{align*}
		(\mathbf{I} - h\vM) \vx_{n+1} &= \vx_n + h\vN\vu_n + \vsigma(t_n)\Delta\vw_n,
	\end{align*}
	where $\mathbf{I}$ is the identity matrix. Finally, solving for $\vx_{h}\in\vX_{h}$ yields:
	\begin{align}\label{forward difference equations}
		\begin{cases}
			\vx_{n+1} = (\mathbf{I} - h\vM)^{-1} \left(\vx_n + h\vN\vu_n + \vsigma(t_n) \Delta\vw_n\right), &\forall n = 0, 1, \ldots, N-1, \\
			\vx_0 = \vx_0.
		\end{cases}
	\end{align}
	This provides the update rule for the discrete state variable at each time step.
	
	Next, consider a given discrete state variable $\vx_{h} \equiv \{\vx_n\}_{n=1}^{N} \in \vX_{h}$. Using the implicit Euler scheme, we approximate the adjoint variable $\vp_{h} \equiv \{\vp_n\}_{n=0}^{N}\in\vX_h$, which satisfies the following backward difference equations:
	\begin{align}\label{backward difference equations}
		\begin{cases}
			\vp_n = (\mathbf{I}-h\vM)^{-1}\mathbb{E}\left[\vp_{n+1} - \vB\vx_{n+1} \big| \mathbb{F}_{t_n}\right] & \forall\, n = N-1, \ldots, 0, \\
			\vp_N = -\vD\vx_N.
		\end{cases}
	\end{align}
	This equation provides the update rule for the adjoint variable, working backwards from the terminal condition at $n = N$.
	
	We also define the operator $\mathcal{T}_{h} : \vX_{h} \to \vX_{h}$, which maps the discrete state variable $\vx_{h}$ to the discrete adjoint variable $\vp_{h}$, such that
	\[
	\mathcal{T}_{h}[\vx_{h}] = \vp_{h}[\vx_{h}],
	\]
	where $\vp_{h}[\vx_{h}]$ denotes the unique solution to the system of equations given by \eqref{backward difference equations}.
	
	The operators $\mathcal{S}_{h}$ and $\mathcal{T}_{h}$ thus provide the mappings between the control and state variables, and between the state and adjoint variables, respectively, within the discretized time framework.  As a discrete variant of \eqref{random differential equation}, for any $\vu_h\in\vU_h$, $\vy_h[\vu_h]\in\vX_h$ denotes the unique solution to the following system of random difference equations
	\begin{align}\label{discrete random differential equation}
		\begin{cases}
			\vy_{n+1} = (\mathbf{I} - h\vM)^{-1} \left(\vy_n + h\vN\vu_n\right) &\forall n = 0, 1, \ldots, N-1, \\
			\vy_0 = 0.
		\end{cases}
	\end{align}  
	We define $\vx_h^\sigma=\vx_h[0]$. Note that $\vx_h[\vu_h]=\vy_h[\vu_h]+\vx_h^\sigma$ for all $\vu_h\in\vU_h$.
	We now define the operator $\mathcal{S}_{h} : \vU_{h} \to \vX_{h}$, which maps the discrete control $\vu_{h}$ to the discrete state $\vx_{h}$, such that
	\[
	\mathcal{S}_{h}(\vu_{h}) = \vy_{h}[\vu_{h}],
	\]
	where $\vy_{h}[\vu_{h}]$ denotes the unique solution to the system of equations given by \eqref{discrete random differential equation}. It is easy to verify
	\begin{align}\label{today 200}
		\|\mathcal{S}_{h}(\vu_{h})\|_{\vX_{h}}^2\le\,T\|\vN\vu_{h}\|^2_{\vU_{h}}.
	\end{align}
	\subsection{Discrete optimal control problem}
	Let us define the cost functional as follows:
	\begin{align*}
		\mathcal{J}_{h}(\vx_{h},\vu_{h}) := \frac{1}{2}\mathbb{E}\left[\int_{0}^{T} \left( \left<\vx_{h}(s), \vB\vx_h(s)\right>+ \alpha\left< \vu_{h}(s),\vu_{h}(s)\right> \right) \,{\rm d}s + \left<\vx_{h}(T),\vD \vx_{h}(T)\right>\right],
	\end{align*}
	where the goal is to find an optimal pair $(\vx_{h}^*, \vu_{h}^*)\in\vX_{h}\times\vU_{h}$ that satisfies
	\begin{align}\label{discrete cost function}
		\mathcal{J}_{h}(\vx_{h}^*, \vu_{h}^*) = \min_{(\vx_{h}, \vu_{h}) \in \vX_h\times\vU_{ad}^h} \mathcal{J}_{h}(\vx_{h}, \vu_{h}),
	\end{align}
	subject to the following forward difference equations:
	\begin{align} \label{forward_difference_equations_state}
		\begin{cases}
			\vx_{n+1} = (\mathbf{I} - h\vM)^{-1} \left(\vx_n + h\vN\vu_n + \vsigma(t_n) \Delta\vw_n\right) &\forall n = 0, 1, \ldots, N-1, \\
			\vx_0 = \vx_0.
		\end{cases}
	\end{align}
	We refer this discrete stochastic optimal control problem \eqref{discrete cost function}-\eqref{forward_difference_equations_state} as $\mathbf{SLQ}_{h}$ problem for short.
	\begin{Theorem}\label{discrete PMP}
		For each $h > 0$, there exists a unique discrete optimal pair $(\vx_{h}^*, \vu_{h}^*) \in \vX_h\times\vU_{ad}^h$ that minimizes $\mathbf{SLQ}_h$ problem such that
		\begin{align}\label{today11}
			\|\sqrt{B}\vx_h^*\|_{\mathbb{L}^2}^2+ \|\sqrt{\alpha}\vu_h^*\|_{\mathbb{L}^2}^2+\mathbb{E}\big[\|\sqrt{\vD}\vx_h^*(T)\|^2\big]\le C_T\big(\|\vx_0\|^2+ \|\sigma\|_{C_t\mathbb{L}^2_\omega}^2\big).
		\end{align} Moreover, the optimality condition is given by:
		\begin{align}\label{discrete optimality condition}
			\vu_{h}^* =\mathcal{P}_{ad}[ \frac{1}{\alpha}\vN^{T}\vp_{h}[\vx_{h}^*]],
		\end{align}
		where $\vp_{h}[\vx_{h}^*]$ is  the unique solution to the backward difference equations \eqref{backward difference equations}. Moreover, the following integral inequality holds: for all $\vu_h\in{\vU}_{ad}^h$
		\begin{align}\label{today07}
			\mathbb{E}\bigg[\int_0^T\left<\vy_h[\vu_h-\vu^*_h](s),\vB\vx_h^*(s)\right>+\alpha\left<\vu_h(s)-\vu^*_h,\vu_h^*(s)\right>{\rm d}s + \left<\vy_h[\vu_h-\vu^*_h](T),\vD\vx_h^*(T)\right>\bigg]\ge\,0.
		\end{align}
	\end{Theorem}
	
	\begin{proof}
		For the proof, one can follow similar lines as done in the proof of Theorem \ref{Lemma_optimality condition}. This is discrete variant of Theorem  \ref{Lemma_optimality condition}. We leave the details to the interested reader. 
	\end{proof}
	\begin{Remark}
		By following the same lines as in the  proof of Theorem \ref{Lemma_optimality condition}, we can compute the Fr\'echet derivative of reduced discrete cost functional $\mathcal{\hat{J}}_h(\vu_h)=\mathcal{J}(\vx_h[\vu_h],\vu_h)$ with the help of adjoint operator $\mathcal{T}_h$ as follows: for all $\vu_h\in\vU_h$,
		\begin{align*}
			\hat{\mathcal{J}}'(\vu_{h})=\alpha\vu_{h}- \vN^{T}\mathcal{T}_{h}[\vu_{h}].
		\end{align*}
		This identity will help to compute gradient iterates in the projected gradient descent method. 
	\end{Remark}
	\begin{Remark}[characterization of the optimal pair] By putting together the results in Theorem \eqref{discrete PMP} and \eqref{backward difference equations}, we arrive at the following optimality system.
		The optimal pair $(\vx_h^*, \vu_h^*)\in \vX_h\times\vU_{ad}^h$ is characterized by the following decoupled system of equations
		\begin{align}\label{characterization of the optimal pair}
			\begin{cases}
				\vx_{n+1}^* = (I - h\vM)^{-1} \left(\vx_n^* + h\vN\vu_n^* + \vsigma(t_n) \Delta\vw_n\right) &\forall\, n = 0, 1, \ldots, N-1, \\
				\vp_n =  (\mathbf{I}-h\vM)^{-1} \mathbb{E}\left[\vp_{n+1} - \vB \vx_{n+1}^*\big| \mathbb{F}_{t_n}\right] & \forall\,n = N-1, \ldots, 0, \\
				\vu_{h}^* =\mathcal{P}_{ad}[\frac{1}{\alpha} \vN^{T}\vp_{h}[\vx_{h}^*]],\\
				\vx^*_0 = \vx_0, \\
				\vp_N = -\vD\vx^*_N.
			\end{cases}
		\end{align}
	\end{Remark}
	\begin{Remark}
		The above system provides a complete characterization of the optimal pair in the discrete setting. The backward recursion for $\vp_n$ is particularly important as it determines the optimal control process through the optimality condition.
	\end{Remark}
	From this characterization \eqref{characterization of the optimal pair}, this is a decoupled system with conditional expectation whose direct computation is out of reach. To address this issue, we approximate to optimal pair $(\vx_{h}^*,\vu_{h}^*)$ by an {\em implementable} scheme in the context of a projected gradient descent method which accounts the control constraints as well.
	\subsection{Projected gradient descent method}
	By \eqref{characterization of the optimal pair}, solving minimization $\mathbf{SLQ}_{h}$ problem is equivalent to solving the system of {\em coupled} forward-backward difference equations \eqref{characterization of the optimal pair}. We may exploit the variational character of $\mathbf{SLQ}_{h}$ problem to apply the projected gradient descent method ($\mathbf{SLQ}_{h}^{{\rm grad}}$, for short) where approximate iterates of the optimal control $\vu^*_{h}$ in a convex set of the Hilbert space $\vU_{h}$ are obtained. 
	
	\begin{algorithm}[H]
		\caption{Projected gradient descent method $(\mathbf{SLQ}_{h}^{{\rm grad}})$}
		\label{tt}
		\begin{algorithmic}[1]
			\State \textbf{Input:} Fix $\vx_0\in\mathbb{R}^d$, initial guess $\vu_{h}^{(0)}\in \vU_{ad}^h$, and fix $\kappa>0$.
			\State \textbf{Iterates:} For any $\ell\in \mathbb{N}\cup\{0\}$;
			\State \textbf{State iterates:} Find the state iterates $\vx_{h}^{(\ell)}\in\vX_{h}$ as follows:
			\begin{align} 
				\begin{cases}
					\vx_{n+1}^{(\ell)} = (I - h\vM)^{-1} \left(\vx_n^{(\ell)} + h\vN\vu_n^{(\ell)} + \vsigma(t_n) \Delta\vw_n\right) &\forall n = 0, 1, \ldots, N-1, \\
					\vx_0^{(\ell)} = \vx_0.
				\end{cases}
			\end{align}
			\State \textbf{Adjoint iterates:} Find the adjoint iterates $\vp_{h}^{(\ell)}\in\vX_{h}$ as follows:
			\begin{align}\label{backward difference equations_iterate}
				\begin{cases}
					\vp_n^{(\ell)} = h\vM\vp_{n}^{(\ell)} + \mathbb{E}\left[\vp_{n+1}^{(\ell)} - \vB\vx_{n+1}^{(\ell)} \big| \mathbb{F}_{t_n}\right] & \forall n = N-1, \ldots, 0, \\
					\vp_N^{(\ell)} =- \vD \vx_N^{(\ell)}.
				\end{cases}
			\end{align}
			\State \textbf{Update iterates:}
			Update $\vu_{h}^{(\ell+1)}\in {\vU}_{ad}^h$ by the following formula:
			\begin{align}\label{control update}
				{\vu}_{h}^{(\ell+1)}=\mathcal{P}_{ad}[ \vu_{h}^{(\ell)}-\frac{1}{\kappa}( \alpha\vu_{h}^{(\ell)}-\vN^{T}\vp_{h}^{(\ell)})].
			\end{align}
			
		\end{algorithmic}
	\end{algorithm}
	
	At this form of projected gradient descent method is not {\em implementable} due to the presence of the conditional expectation. But in next proposition, we will see that this conditional expectation can be avoid by using the iterate nature of gradient descent method and the linearity of the state-adjoint equations, this strategy was before exploited in the case of the stochastic Dirichlet boundary optimal control problem \cite{Chaudhary}.  
	\subsection{Recursive formula}
	In this subsection, we derive a recursive formula for exact computation of conditional expectation. 
	\begin{Proposition}\label{Proposition recursive formula}
		Let $\vA_0=(\mathbf{I}-h\vM)^{-1}$, a deterministic constant $\vu_h^{(0)}\equiv \mathfrak{c}\in\Pi_{i=1}^{M}[a_i, b_i]$. For any $\ell\in\mathbb{N}$, in \eqref{backward difference equations_iterate}, the adjoint $\vp_n^{(\ell)}$ can be computed by the following recursive formula:
		\begin{align}\label{recursive formula}
			\vp_n^{(\ell)}&=-\vA_0^{N-n}\vD\vx_n^{(\ell)}-h\sum_{q=n+1}^{N}\vA_0^{q-n}\vB\vA_0^{q-n}\vx_n^{(\ell)}-h\vA_0^{N-n}\vD\sum_{r=n}^{N-1}\vA_0^{N-r}\vN \mathfrak{U}_{r,n}^{(\ell)}\notag\\&\qquad\qquad- h^2\sum_{q=n+1}^N\vA_0^{q-n}\vB \sum_{r=n}^{q-1}\vA_0^{q-r}\vN\mathfrak{U}_{r,n}^{(\ell)}.
		\end{align}
		where $\mathfrak{U}_{r,n}^{(\ell)}$ is defined by the following recursive formula: for all $\ell\in\mathbb{N}$, for all $r\in\{n+1,...,N-1\}$, $\mathbb{P}$-almost surely
		\begin{align}
			\mathfrak{U}^{(\ell)}_{r,n}=&\mathcal{P}_{ad}\bigg[\big(\mathbf{I}-\frac{\alpha}{\kappa}\big)\mathfrak{U}_{r,n}^{(\ell-1)}-\frac{1}{\kappa}\vN^{T}\vA_0^{N-r}\vD\bigg[\vA_0^{N-n}\vx_n^{(\ell-1)}+ h\sum_{m=n+1}^{N-1}\vA_0^{N-m}\vN\mathfrak{U}_{m,n}^{(\ell-1)}\bigg]\notag\\&\qquad-\frac{h}{\kappa}\vN^{T}\sum_{s=r+1}^{N}\vA_0^{s-r}\vB\bigg[\vA_0^{s-n}\vx_n^{(\ell-1)}+h\sum_{m=n+1}^{s-1}\vA_0^{s-m}\vN\mathfrak{U}_{m,n}^{(\ell-1)}\bigg]\bigg],
		\end{align}
		with $\mathfrak{U}_{r,n}^{(0)}=\mathfrak{c}\quad\forall\,r\in\{n+1,....,N-1\}$.
	\end{Proposition}
	\begin{proof} We give complete the proof in the several steps.
		
		\noindent
		\textbf{Step 1.} By using the tower property of conditional expectation and linearity of adjoint iterate \eqref{backward difference equations_iterate}, for any $n\in\{1,...,N-1\}$,
		\begin{align*}
			\vp_n^{(\ell)}=-\mathbb{E}\big[A_0^{N-n}\vD\vx_N^{(\ell)} +h\sum_{q=n+1}^{N}\vA_0^{q-n}\vB\vx_q^{(\ell)}\Big|\mathbb{F}_{t_n}\big].
		\end{align*}
		\textbf{Step 2.} By making use of the state iterate \eqref{forward_difference_equations_state} and additive nature of noise coefficients, we can conclude that for all $q>n$,
		\begin{align*}
			\vx_q^{(\ell)}=\vA_0^{q-n}\vx_n^{(\ell)}+h\sum_{r=n}^{q-1}\vA_0^{q-r}[\vN \vu_r^{(\ell)}+\vsigma_r\Delta_{r}\vw].
		\end{align*}
		It implies that
		\begin{align*}
			\mathbb{E}\big[\vx_q^{(\ell)}\big|\mathbb{F}_n\big]=\vA_0^{q-n}\vx_n^{(\ell)}+ h\sum_{r=n}^{q-1}\vA_0^{q-r}\vN\mathfrak{U}_{r,n}^{(\ell)},
		\end{align*}
		where $\mathfrak{U}_{r,n}^{(\ell)}=\mathbb{E}[\vu_r^{(\ell)}\big|\mathbb{F}_n]$.
		
		\noindent
		\textbf{Step 3.} With the help of the control iterate \eqref{control update}, we can conclude that
		\begin{align*}
			\mathfrak{U}^{(\ell)}_{r,n}=\mathcal{P}_{ad}\bigg[(1-\frac{\alpha}{\kappa})\mathfrak{U}_{r,n}^{(\ell-1)}-\frac{1}{\kappa}\vN^{T}\mathbb{E}\bigg[\vA_0^{N-r}\vD\vx_N^{(\ell-1)}+h\sum_{s=r+1}^N\vA_0^{s-r}\vB\vx_s^{(\ell-1)}\Big|\mathbb{F}_n\bigg]\bigg].
		\end{align*}
		Analogously to Step 2, we know that for all $s\in\{n+2,...,N\}$, $\mathbb{P}$-almost surely,
		\begin{align*}
			\mathbb{E}\big[\vx_{s}^{(\ell-1)}\big|\mathbb{F}_n\big]=\vA_0^{s-n}\vx_n^{(\ell-1)}+ h\sum_{m=n+1}^{s-1}\vA_0^{s-m}\vN\mathfrak{U}_{m,n}^{(\ell-1)}.
		\end{align*}
		From the above identity, we conclude that
		\begin{align*}
			\mathfrak{U}^{(\ell)}_{r,n}=&\mathcal{P}_{ad}\bigg[\big({1}-\frac{\alpha}{\kappa}\big)\mathfrak{U}_{r,n}^{(\ell-1)}-\frac{1}{\kappa}\vN^{T}\vA_0^{N-r}\vD\bigg[\vA_0^{N-n}\vx_n^{(\ell-1)}+ h\sum_{m=n+1}^{N-1}\vA_0^{N-m}\vN\mathfrak{U}_{m,n}^{(\ell-1)}\bigg]\\&\qquad-\frac{h}{\kappa}\vN^{T}\sum_{s=r+1}^{N}\vA_0^{s-r}\vB\bigg[\vA_0^{s-n}\vx_n^{(\ell-1)}+h\sum_{m=n+1}^{s-1}\vA_0^{s-m}\vN\mathfrak{U}_{m,n}^{(\ell-1)}\bigg]\bigg],
		\end{align*}
		and
		\begin{align*}
			\mathfrak{U}_{r,n}^{(\ell)}=\mathfrak{c}\qquad	\forall\,r\in\{n+1,...,N-1\}.
		\end{align*}
		\textbf{Step 4.} With the information of Step 1 and Step 2, we can conclude that $\mathbb{P}$-almost surely,
		\begin{align*}
			\vp_n^{(\ell)}&=-\vA_0^{N-n}\vD\bigg[\vx_n^{(\ell)}+ h\sum_{r=n}^{N-1}\vA_0^{N-r}\vN\mathfrak{U}_{r,n}^{(\ell)}\bigg]-h\sum_{q=n+1}^N\vA_0^{q-n}\vB\bigg[\vA_0^{q-n}\vx_n^{(\ell)}+ h\sum_{r=n}^{q-1}\vA_0^{q-r}\vN\mathfrak{U}_{r,n}^{(\ell)}\bigg]\\&=-\vA_0^{N-n}\vD\vx_n^{(\ell)}-h\sum_{q=n+1}^{N}\vA_0^{q-n}\vB\vA_0^{q-n}\vx_n^{(\ell)}-h\vA_0^{N-n}\vD\sum_{r=n}^{N-1}\vA_0^{N-r}\vN \mathfrak{U}_{r,n}^{(\ell)}\\&\qquad\qquad- h^2\sum_{q=n+1}^N\vA_0^{q-n}\vB \sum_{r=n}^{q-1}\vA_0^{q-r}\vN\mathfrak{U}_{r,n}^{(\ell)}.
		\end{align*}
		This completes the proof.
	\end{proof}
	\newpage
	\subsection{An efficient {\em implementable} scheme}
	By using the recursive formula \eqref{recursive formula}, we can rewrite \eqref{backward difference equations_iterate} of Algorithm \ref{tt} to obtain an {\em implementable} scheme in Algorithm \ref{Algorithm_implementable}.
	
	{
		\begin{algorithm}[H]
			\caption{Implementable Algorithm}
			\label{Algorithm_implementable}
			\begin{algorithmic}[1]
				\State \textbf{Input:} Fix parameter $\vx_0\in\mathbb{R}^d$, $h>0$, initial guess $\vu_{h}^{(0)}\equiv \mathfrak{c}\in \Pi_{i=1}^m[a_i,b_i]$, $\kappa>0$.
				\State \textbf{Iterates:} For any $\ell\in \mathbb{N}\cup\{0\}$;
				\State \textbf{State iterates:} Find the state iterates $\vx_{h}^{(\ell)}\in\vX_{h}$ as follows:
				\begin{align*} 
					\begin{cases}
						\vx_{n+1}^{(\ell)} = (I - h\vM)^{-1} \left(\vx_n^{(\ell)} + h\vN\vu_n^{(\ell)} + \vsigma(t_n) \Delta\vw_n\right), &\forall n = 0, 1, \ldots, N-1, \\
						\vx_0^{(\ell)} = \vx_0.
					\end{cases}
				\end{align*}
				\State \textbf{Recursive formula}: $\mathfrak{U}_{r,n}^{(\ell)}$ is defined by recursive formula: for all $\ell\in\mathbb{N}$, for all $r\in\{n+1,...,N-1\}$, $\mathbb{P}$-almost surely
				\begin{align*}
					\mathfrak{U}^{(\ell)}_{r,n}=&\mathcal{P}_{ad}\bigg[\big(1-\frac{\alpha}{\kappa}\big)\mathfrak{U}_{r,n}^{(\ell-1)}-\frac{2}{\kappa}\vN^{T}\vA_0\vD\bigg[\vA_0^{N-n}\vx_n^{(\ell-1)}+ h\sum_{m=n+1}^{N-1}\vA_0^{N-m}\vN\mathfrak{U}_{r,n}^{(\ell-1)}\bigg]\\&\qquad+\frac{1}{\kappa}\sum_{s=r+1}^{N}\vA_0^{s-r}\bigg[\vA_0^{s-n}\vx_n^{(\ell-1)}+h\sum_{m=n+1}^{s-1}\vA_0^{s-m}\vN\mathfrak{U}_{m,n}^{(\ell-1)}\bigg]\bigg],
				\end{align*}
				with $\mathfrak{U}_{r,n}^{(0)}=\mathfrak{c}\qquad\forall\,r\in\{n+1,....,N-1\}$.
				
				\State\textbf{Adjoint iterates:} Find the adjoint iterates $\vp_{h}^{(\ell)}\in\vX_{h}$ as follows:
				\begin{align*}
					\vp_n^{(\ell)}&=-\vA_0^{N-n}\vD\vx_n^{(\ell)}-h\sum_{q=n+1}^{N}\vA_0^{q-n}\vB\vA_0^{q-n}\vx_n^{(\ell)}-h\vA_0^{N-n}\vD\sum_{r=n}^{N-1}\vA_0^{N-r}\vN \mathfrak{U}_{r,n}^{(\ell)}\notag\\&\qquad\qquad- h^2\sum_{q=n+1}^N\vA_0^{q-n}\vB \sum_{r=n}^{q-1}\vA_0^{q-r}\vN\mathfrak{U}_{r,n}^{(\ell)}.
				\end{align*}	
				\State\textbf{Update iterates:}
				Update $\vu_{h}^{(\ell+1)}\in {\vU}_{ad}^h$ by the following formula:
				\begin{align*}
					{\vu}_{h}^{(\ell+1)}=\mathcal{P}_{ad}[ \vu_{h}^{(\ell)}-\frac{1}{\kappa}(\alpha \vu_{h}^{(\ell)}-\vN^{T}\vp_{h}^{(\ell)})].
				\end{align*}
			\end{algorithmic}
	\end{algorithm}}
	\subsection{Numerical example: simulation of solution iterates}\label{example} In this example, we present a numerical simulation of {\em Algorithm \ref{Algorithm_implementable}} with a relatively high dimension, \( d = 10 \). While the algorithm can be implemented for any higher dimension. We have chosen \( d = 10 \) to ensure clarity in the plots and clearly interpretation of the results.
	Let \(\vB = \vD = \mathbf{I}\) (identity matrix), and \(\vM\), \(\vN\) be randomly generated fixed matrices. Define the noise level as 
	\[
	\sigma(t_n) = 0.5 \sin\left(\frac{nT\pi}{N}\right) \sigma_{1,n}, \quad n = 0,1, \ldots, N-1,
	\]
	where \(\sigma_{1,n}\) is a randomly generated fixed matrix for all $n=0,1,...,N-1$. Set parameters: \(T = 0.4\), \(L = 10\), \(N = 20\), \(\kappa = 0.45\), \(\alpha = 0.04\), $d=10$, $m=k=4$, $\mathtt{M}=1000$, $\mathfrak{c}=0$. The constraints are fixed as \(a_i = -2\) and \(b_i = 2\) for all \(1 \leq i \leq d\).
	
	\begin{figure}[H]
		\centering
		\begin{subfigure}{0.45\textwidth}
			\centering
			\includegraphics[width=\linewidth]{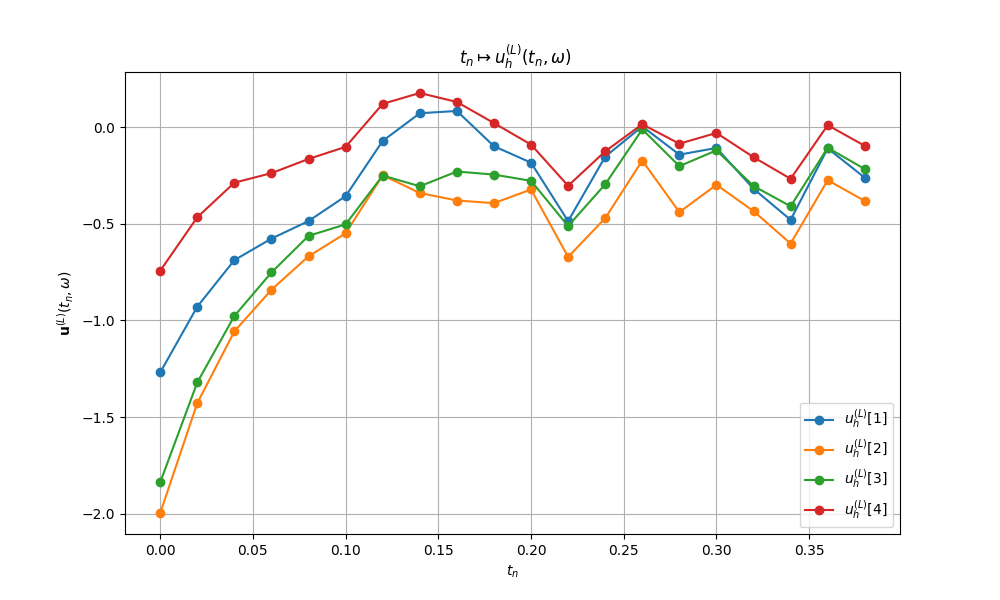}
			\caption{$t_n\to \vu_h^{(L)}(t_n,\omega)$}
			\label{fig:image1}
		\end{subfigure}
		\hfill
		\begin{subfigure}{0.45\textwidth}
			\centering
			\includegraphics[width=\linewidth]{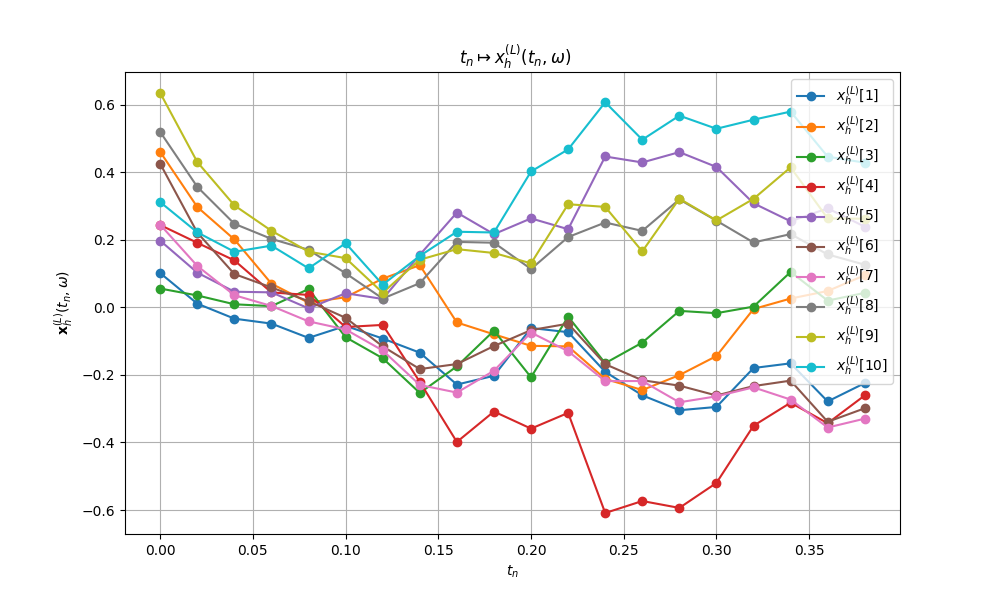}
			\caption{$t_n\to\vx_h^{(L)}(t_n,\omega)$}
			\label{fig:image2}
		\end{subfigure}
		\caption{\textbf{(a)} and \textbf{(b)} display the plots of the control iterate \(\vu_h^{(L)}\) and state iterate \(\vx_h^{(L)}\), respectively, in the case of {\em control constraint}.}
	\end{figure}
	\begin{figure}[H]
		\begin{subfigure}{0.45\textwidth}
			\centering
			\includegraphics[width=\linewidth]{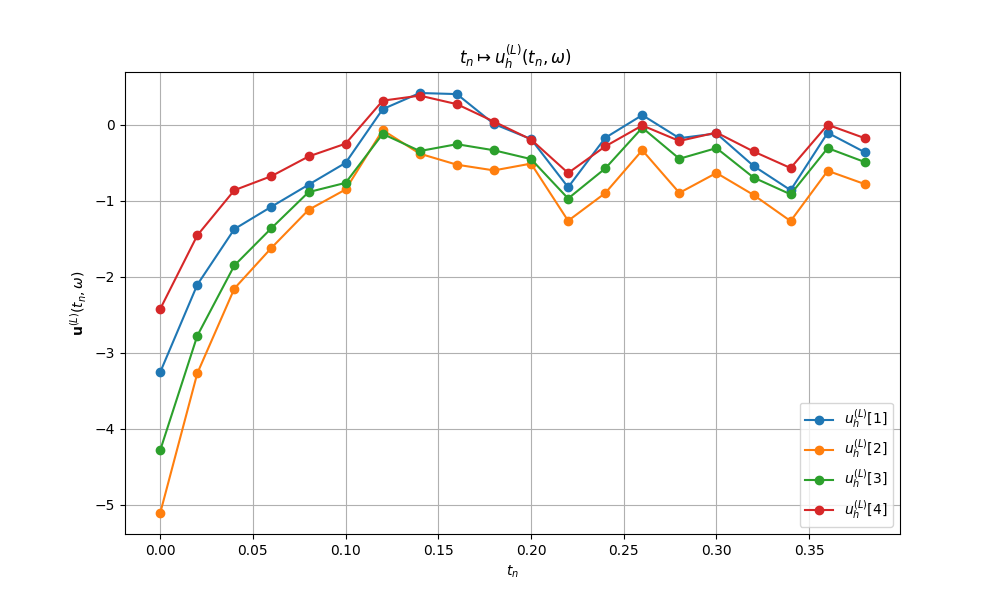}
			\caption{$t_n\to \vu_h^{(L)}(t_n,\omega)$}
			\label{fig:image3}
		\end{subfigure}
		\hfill
		\begin{subfigure}{0.45\textwidth}
			\centering
			\includegraphics[width=\linewidth]{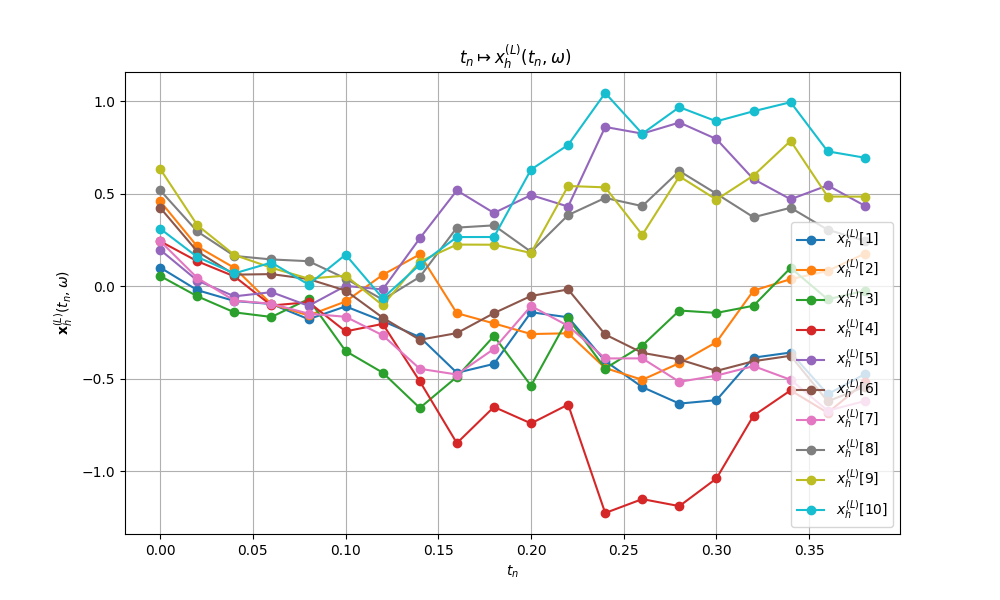}
			\caption{$t_n\to\vx_h^{(L)}(t_n,\omega)$}
			\label{fig:image4}
		\end{subfigure}
		\caption{\textbf{(a)} and \textbf{(b)} present the plots of the control iterate \(\vu_h^{(L)}\) and state iterate \(\vx_h^{(L)}\), respectively, in the case of free control.}
	\end{figure}
	
	\section{Error analysis}\label{section 4}
	In this subsection, we discuss the error analysis related to time discretization. Our focus is on estimating the errors introduced when approximating the continuous-time control and state trajectories using their discrete counterparts. 
	
	\subsection{Error bounds for state and adjoint differential equations}
	We begin by stating a result that provides an error bound for the state stochastic differential equation (SDE) when discretized.
	
	\begin{Proposition}\label{Corollary 4.1}
		Let $\vu_h \in \vU_h$. Let $\vy[\vu_h]$ and $\vy_h[\vu_h]$ be the unique solutions to \eqref{random differential equation} and \eqref{discrete random differential equation}, respectively. Then the following error bounds hold:
		\begin{align}\label{today04}
			\sup_{0\le\,n\le N}\mathbb{E}\big[\|\vy[\vu_h](t_n) - \vy_{h}[\vu_h](t_n)\|^2\big] \le C_{T,d} h^2\big\|\vu_h\|_{\mathbb{L}^2}^2,
		\end{align}
		and
		\begin{align}\label{today05}
			\|\vy[\vu_h]-\vy_h[\vu_h]\|_{\mathbb{L}^2}^2\le\,C_{T,d}h^2\|\vu_{h}\|_{\mathbb{L}^2}^2.
		\end{align}
	\end{Proposition}
	\begin{proof}
		The proof is standard. For completeness of the proof, we give details here. For simplicity, we denote \(\vy[\vu_h] \equiv \vy\). Define the error term for all \(0 \leq n \leq N\) as
		\[
		e_n = \vy(t_n) - \vy_n.
		\]
		Then, we have \(\mathbb{P}\)-a.s., for all \(0 \leq n \leq N-1\),
		\[
		e_{n+1} = e_n + \int_{t_n}^{t_{n+1}} \vM (\vy(s) - \vy_{n+1}) \, \mathrm{d}s.
		\]
		Taking the inner product with \(e_{n+1}\), we obtain
		\[
		\left\langle e_{n+1} - e_n, e_{n+1} \right\rangle = -\int_{t_n}^{t_{n+1}} \left\langle \vM_{1}(\vy(s) - \vy(t_{n+1})), \vM_1 e_{n+1} \right\rangle \, \mathrm{d}s - h \|\vM_1 e_{n+1}\|^{2}.
		\]
		Applying the Cauchy-Schwarz inequality and Young’s inequality to the first term on the right-hand side, we have
		\[
		\left\langle e_{n+1} - e_n, e_{n+1} \right\rangle \leq -\frac{1}{2} h \|\vM_1 e_{n+1}\|^2 + \frac{1}{2} \int_{t_n}^{t_{n+1}} \|\vM_1(\vy(s) - \vy(t_{n+1}))\|^2 \, \mathrm{d}s.
		\]
		This implies
		\[
		\mathbb{E} \big[\|e_{n+1}\|^2 \big] + \mathbb{E}\big[\|e_n - e_{n+1}\|^2\big] + \frac{h}{2} \mathbb{E} \big[\|\vM_1 e_{n+1}\|^2 \big] \leq \mathbb{E} \big[\|e_n\|^2 \big] + \frac{1}{2} \mathbb{E} \bigg[\int_{t_n}^{t_{n+1}} \|\vM_1(\vy(s) - \vy(t_{n+1}))\|^2 \, \mathrm{d}s \bigg].
		\]
		By summing over \(n\) from \(0\) to \(N-1\), we obtain, for all \(0 \leq n \leq N\),
		\[
		\mathbb{E} \big[\|e_n\|^2\big] + \frac{h}{2} \sum_{k=1}^{n} \mathbb{E} \big[\|\vM_1 e_k\|^2 \big] \leq \mathbb{E} \big[\|e_0\|^2\big] + \frac{1}{2} \sum_{k=0}^{n-1} \mathbb{E} \bigg[\int_{t_k}^{t_{k+1}} \|\vM_1(\vy(s) - \vy(t_{k+1}))\|^2 \, \mathrm{d}s \bigg].
		\]
		Using the time regularity estimate \eqref{lll}, we obtain
		\[
		\sup_{1 \leq n \leq N} \mathbb{E} \big[\|e_n\|^2 \big] + \frac{h}{2} \sum_{n=1}^{N} \mathbb{E} \big[\|\vM_1 e_n\|^2 \big] \leq C_{T,d} h^2 \|\vu_h\|_{\mathbb{L}^2}^2.
		\]
		This yields the estimate \eqref{today04}. By applying the estimate \eqref{today04} and noting that \(\|\vy_{t}\|_{\mathbb{L}^2}^2 \le C_{T} \|\vu_h\|^2_{\mathbb{L}^2}\), we can similarly derive the estimate \eqref{today05}.
	\end{proof}
	
	\begin{Proposition}
		Let $\vx^\sigma$ be the unique solution to \eqref{sigmasde}. Let $\sigma$ satisfies the following estimate:
		\begin{align}\label{today08}
			\|\sigma-\Pi_h\sigma\|_{\mathbb{L}^2}^2\le\,C_T\,h.
		\end{align} Then the following error bound holds:
		\begin{align}\label{today09}
			\sup_{0\le n\le N}\|\vx^\sigma(t_n)-\vx^\sigma_h(t_n)\|_{\mathbb{L}^2_\omega}^2 +\|\vx^\sigma-\vx_h^\sigma\|_{\mathbb{L}^2}^2\le\,C_T\,h(\|\vM\vx_0\|^2+\|\vM \sigma\|_{\mathbb{L}^2}^2).
		\end{align}
	\end{Proposition}
	\begin{proof}
		The	proof of this result can be shown as follows similar lines as done in the previous Proposition \ref{Corollary 4.1} with the help of the estimates \eqref{gradient stability for forward SDE} \& \eqref{today08} and It\^o isometry.
	\end{proof}
	\subsection{Error estimate for optimal pair}
	\begin{Theorem}[Rate of convergence for control constraint]\label{theorem 4.1} Let $\vx_0\in\mathbb{R}^d$. Let $\sigma$ satisfies the following estimate
		\begin{align*}
			\|\sigma-\Pi_h\sigma\|_{\mathbb{L}^2}^2\le\,C_T\,h.
		\end{align*}
		Let $(\vx^*, \vu^*) \in \vX \times \vU_{ad}$ and $(\vx_{h}^*, \vu_{h}^*)\in\vX_h\times\vU_{ad}^h$ be the unique optimal pairs for the continuous $\mathbf{SLQ}$ problem \eqref{cost functional minimum}-\eqref{forward SDE} and the discrete $\mathbf{SLQ}_h$ problem \eqref{discrete cost function}-\eqref{forward_difference_equations_state}, respectively. Then the following error bound holds:
		\begin{align}\
			\|\sqrt{\alpha}(\vu^*-\vu_h^*)\|_{\mathbb{L}^2}^2&+ \|\sqrt{\vB}(\vx^*-\vx_h^*)\|_{\mathbb{L}^2}^2 + \mathbb{E}\big[\|\sqrt{\vD}(\vx^*(T)-\vx_h^*(T))\|^2\big]\notag\\&\le C_{T,d}\,h^{1/2}(\|\vx_0\|^2+ \|\vM_1\vx_0\|^2+\|\vM\vx_0\|^2 + \|\sigma\|_{C_t\mathbb{L}^2}^2+\|\vM\vsigma\|_{\mathbb{L}^2}^2+ \|\vM_1\vsigma\|_{\mathbb{L}^2}^2).
		\end{align}
	\end{Theorem}
	\begin{proof}
		We start with the expression for the squared norm of the difference between the optimal control and its approximation:
		\begin{align*}
			\|\sqrt{\alpha}(\vu^* - \vu_h^*)\|_{\mathbb{L}^2}^2 := I + J,
		\end{align*}
		where 
		\begin{align*}
			I &= \mathbb{E}\bigg[\int_0^T \left<\alpha(\vu^* - \vu_h^*), \Pi_h\vu^* - \vu_h^*\right> {\rm d}s\bigg], \\
			J &= \mathbb{E}\bigg[\int_0^T \left<\alpha(\vu^* - \vu_h^*), \vu^* - \Pi_h\vu^*\right> {\rm d}s\bigg].
		\end{align*}
		\textbf{Step 1.} In this step, we will estimate the term \( I \):
		\begin{align*}
			I &= \mathbb{E}\bigg[\int_0^T \left<\alpha\vu^*, \Pi_h\vu^* - \vu_h^*\right> {\rm d}s\bigg] - \mathbb{E}\bigg[\int_0^T \left<\alpha\vu_h^*, \Pi_h\vu^* - \vu_h^*\right> {\rm d}s\bigg] \\
			&= K_1 + K_2.
		\end{align*}
		Next, we estimate the term \( K_1 \):
		\begin{align*}
			K_1 &= -\mathbb{E}\bigg[\int_0^T \left<\alpha\vu^*, \vu_h^* - \vu^*\right> {\rm d}s\bigg] + \mathbb{E}\bigg[\int_0^T \left<\alpha\vu^*, \Pi_h\vu^* - \vu^*\right> {\rm d}s\bigg] \\
			&= K_{11} + K_{12}.
		\end{align*}
		To estimate \( K_{11} \), we proceed as follows: By using the integral identity \eqref{today06} and the linearity of the differential equation \eqref{random differential equation}, we obtain:
		\begin{align*}
			K_{11} &\leq \mathbb{E}\bigg[ \int_0^T \left<\vB\vx^*(s), \vy[\vu_h^* - \vu^*](s)\right> {\rm d}s\bigg] + \mathbb{E}\big[\left<\vD\vx_h^*(T), \vy[\vu_h^* - \vu^*](T)\right>\big] \\
			&= \mathbb{E}\bigg[\int_0^T \left<\vB\vx^*(s), \vy[\vu_h^* - \Pi_h\vu^*](s)\right> {\rm d}s\bigg] + \mathbb{E}\big[\left<\vD\vx^*(T), \vy[\vu_h^* - \Pi_h\vu^*](T)\right>\big] \\
			&\quad + \mathbb{E}\big[\int_0^T \left<\vB\vx^*(s), \vy[\Pi_h\vu^* - \vu^*](s)\right> {\rm d}s\bigg] + \mathbb{E}\bigg[\left<\vD\vx^*(T), \vy[\Pi_h\vu^* - \vu^*](T)\right>\big].
		\end{align*}
		This implies that:
		\begin{align*}
			K_{11} + K_{12} &\leq \mathbb{E}\int_0^T \left<\vB\vx^*(s), \vy[\vu_h^* - \Pi_h\vu^*](s)\right> {\rm d}s  + \mathbb{E}\left<\vD\vx^*(T), \vy[\vu_h^* - \Pi_h\vu^*](T)\right> + \mathfrak{E},
		\end{align*}
		where
		\begin{align*}
			\mathfrak{E} &= \mathbb{E}\bigg[\int_0^T \left<\vB\vx^*(s), \vy[\Pi_h\vu^* - \vu^*](s)\right> {\rm d}s\bigg]+ \mathbb{E}\bigg[\int_0^T \left<\vD\vx^*(T), \vy[\Pi_h\vu^* - \vu^*](T)\right>\bigg] \\
			&\quad + \mathbb{E}\bigg[\int_0^T \left<\alpha\vu^*(s), \Pi_h\vu^*(s) - \vu^*(s)\right> {\rm d}s\bigg].
		\end{align*}
		On the other hand, for the term \( K_2 \), we use the integral identity \eqref{today07} to obtain
		\[
		K_2 \le -\mathbb{E}\bigg[\int_0^T \left< \vB \vx_h^*(s), \vy_h[\vu_h^* - \Pi_h \vu^*](s) \right> \, \mathrm{d}s\bigg] - \mathbb{E}\big[\left< \vD \vx_h^*(T), \vy_h[\vu_h^* - \Pi_h \vu^*](T) \right>\big].
		\]
		Now we can conclude that 
		\[
		I = K_1 + K_2 \le I_1 + I_2 + I_3 + I_4 + \mathfrak{E},
		\]
		where the terms are defined as follows:
		\begin{align*}
			I_1 & = \mathbb{E}\left[\int_0^T \left< \vB (\vx^*(s) - \vx_h^*(s)), \vy[\vu_h](s) \right> \, \mathrm{d}s\right], \\
			I_2 & = \mathbb{E}\left[\int_0^T \left< \vB \vx_h^*(s), \vy_h[\vu_h](s) - \vy[\vu_h](s) \right> \, \mathrm{d}s\right], \\
			I_3 & = \mathbb{E}\left[\left< \vD (\vx^*(T) - \vx_h^*(T)), \vy[\vu_h](T) \right>\right], \\
			I_4 & = \mathbb{E}\left[\left< \vD \vx_h^*(T), \vy_h[\vu_h](T) - \vy[\vu_h](T) \right>\right],
		\end{align*}
		with 
		\[
		\vu_h = \vu_h^* - \Pi_h \vu^*.
		\]
		We will now estimate each term separately.

		\noindent
		\textbf{Step 1(a).} In this step, we will estimate $I_1$ term as follows:
		\begin{align*}
			I_1&= -\mathbb{E}\bigg[\int_0^T\left<\vB(\vx^*(s)-\vx^*_h(s)), \vx^*(s)-\vx^*_h(s)\right>{\rm d}s\bigg]-\mathbb{E}\bigg[\int_0^T\left<\vB(\vx^*(s)-\vx^*_h(s)), \vx^\sigma(s)-\vx^\sigma_h(s)\right>{\rm d}s\bigg]\\&\qquad+ \mathbb{E}\bigg[\int_0^T\left<\vB(\vx^*(s)-\vx^*_h(s)), \vy[\vu^*](s)-\vy[\Pi_h\vu^*]\right>{\rm d}s\bigg]\\
			&:=I_{11}+ I_{12} + I_{13}.
		\end{align*}
		Make use of Young's inequality and the estimate \eqref{today09}, we conclude that for any $\delta>0$
		\begin{align*}
			I_{12}&\le \delta\|\sqrt{B}(\vx^*-\vx_h^*)\|_{\mathbb{L}^2}^2 + c_\delta \|\sqrt{B}(\vx^\sigma-\vx_h^\sigma)\|_{\mathbb{L}^2}^2\\&\le\delta\|\sqrt{B}(\vx^*-\vx_h^*)\|_{\mathbb{L}^2}^2  + c_\delta C_T h (\|\vM\vx_0\|^2+\|\vM\sigma\|_{\mathbb{L}^2}^2).
		\end{align*}
		For term $I_{13}$, we obtain with the help of the estimates 
		\eqref{stability for forward SDE_1}\&\eqref{today10}
		\begin{align*}
			I_{13}&\le\,\delta \|\sqrt{\vB}(\vx^*-\vx_h^*)\|_{\mathbb{L}^2}^2 + c_\delta \|\sqrt{\vB}\vy[\vu^*-\Pi_h\vu^*]\|_{\mathbb{L}^2}^2\\&\le\delta \|\sqrt{\vB}(\vx^*-\vx_h^*)\|_{\mathbb{L}^2}^2+c_\delta C_{T,d}\|\vu^*-\Pi_h\vu^*\|^2_{\mathbb{L}^2}\\&\le\,\delta \|\sqrt{\vB}(\vx^*-\vx_h^*)\|_{\mathbb{L}^2}^2 + c_\delta C_{T,d} h(\|\vx_0\|^2+ \|\vM_1\vx_0\|^2 + \|\sigma\|_{\mathbb{L}^2}^2+ \|\vM_1\vsigma\|_{\mathbb{L}^2}^2).
		\end{align*}
		\textbf{Step 1(b).} In this step, we first use \eqref{today05} to estimate $I_2$ as follows:
		\begin{align*}
			I_2&\le\|\vB\vx_h^*\|_{\mathbb{L}^2}\|\vy[\vu_h]-\vy_h[\vu_h]\|_{\mathbb{L}^2}\le \,\|\vB\vx_h^*\|_{\mathbb{L}^2}\,h\|\vu_h\|_{\mathbb{L}^2}\le\,\frac{h}{2}\|\vB\vx_h^*\|_{\mathbb{L}^2}^2 + \frac{h}{2}\|\vu_h\|_{\mathbb{L}^2}^2\\
			&\le\,C_{T,d}h(\|\vx_0\|^2+ \|\sigma\|^2_{C_t\mathbb{L}^2})+ h\big(\|\vu_h^*\|_{\mathbb{L}^2}^2 + \|\Pi_h\vu^*(t)\|_{\mathbb{L}^2}^2\big)\qquad\text{(by using \eqref{today11})}\\
			&\le\,C_{T,d}h(\|\vx_0\|^2+ \|\sigma\|^2_{C_t\mathbb{L}^2}) +C_{T,d}h\sup_{t\in[0,T]}\mathbb{E}\big[\|\vp(t)\|^2\big] \qquad\text{( by using \eqref{optimality condition})}\\&\le\,C_{T,d}h(\|\vx_0\|^2+ \|\sigma\|^2_{C_t\mathbb{L}^2})\qquad\text{(by using \eqref{stability for backward SDE})}.
		\end{align*}
		\textbf{Step 1(c).} In this step, we estimate $I_3$ as follows:
		\begin{align*}
			I_3&=-\mathbb{E}\bigg[\left<\vD(\vx^*(T)-\vx_h^*(T)), \vx^*(T)-\vx_h^*(T)\right>\bigg] - \mathbb{E}\bigg[\left<\vD(\vx^*(T)-\vx_h^*(T)), \vx^\sigma(T)-\vx_h^\sigma(T)\right>\bigg]\\&\qquad +\mathbb{E}\bigg[\left<\vD(\vx^*(T)-\vx_h^*(T)), \vy[\vu^*](T)-\vy[\Pi_h\vu^*](T)\right>\bigg]\\
			&:=I_{31} + I_{32} + I_{33}.
		\end{align*}
		For term $I_{32}$, by making use of Young's inequality and the estimate \eqref{today04}, we obtain
		\begin{align*}
			I_{32}&\le\,\delta\|\sqrt{\vD}(\vx^*(T)-\vx^*_h(T)\|_{\mathbb{L}^2}^2+ c_\delta\|\vx^\sigma(T)-\vx_h^\sigma(T)\|_{\mathbb{L}^2}^2\\
			&\le\,\delta\|\sqrt{\vD}(\vx^*(T)-\vx^*_h(T))\|_{\mathbb{L}^2}^2+ c_\delta C_T h(\|\vM\vx_0\|^2+\|\vM\sigma\|_{\mathbb{L}^2}^2).
		\end{align*}
		For term $I_{34}$, we use \eqref{stability for forward SDE_1}\&\eqref{today11} to get
		\begin{align*}
			I_{33}&\le\,\delta\mathbb{E}\big[\|\sqrt{\vD}(\vx^*(T)-\vx_h^*(T))\|^2\big] + \delta\mathbb{E}\big[\|\vy[\vu^*-\Pi_h\vu^*](T)\|^2\big]\\
			&\le\,\delta\mathbb{E}\big[\|\sqrt{\vD}(\vx^*(T)-\vx_h^*(T))\|^2\big]+ c_\delta C_{T,d}\,h(\|\vx_0\|^2+ \|\vM_1\vx_0\|^2 + \|\sigma\|_{\mathbb{L}^2}^2+ \|\vM_1\vsigma\|_{\mathbb{L}^2}^2).	
		\end{align*}
		\textbf{Step 1(d).} By using similar lines as used in step $1(b)$ and estimates \eqref{today04}\&\eqref{today09}, we obtain
		\begin{align}
			I_{4}\le\,C_\delta h+ \delta\|\vu^*-\vu_h^*\|_{\mathbb{L}^2}^2.
		\end{align}
		Finally we can conclude from last four sub-steps that for any $\delta>0$, there exists $c_\delta>0$ such that
		\begin{align*}
			I_1+ I_2+ I_3+I_4\le(\delta-1)\|\sqrt{\vB}(\vx^*-\vx_h^*)\|_{\mathbb{L}^2}^2 + (\delta-1)\mathbb{E}\big[\|\sqrt{\vD}(\vx^*(T)-\vx_h^*(T))\|^2\big]+ \delta\|(\vu^*-\vu_h^*)\|+ c_\delta C_{T,d} h.
		\end{align*}
		\textbf{Step 1(f).} In this step, we estimate the term $\mathfrak{E}$. We can write 
		\begin{align*}
			\mathfrak{E}=\mathfrak{E}_{1}+\mathfrak{E}_{2}+\mathfrak{E}_3,
		\end{align*}
		where
		\begin{align*}
			\mathfrak{E}_1&=\mathbb{E}\bigg[\int_0^T\left<\vB\vx^*(s),\vy[\Pi_h\vu^*-\vu^*](s)\right>{\rm d}s\bigg],\\
			\mathfrak{E}_2&=\mathbb{E}\bigg[\int_0^T\left<\vD\vx^*(T),\vy[\Pi_h\vu^*-\vu^*](T)\right>\bigg],\\
			\mathfrak{E}_3&=\mathbb{E}\bigg[\int_0^T\left<\alpha\vu^*(s),\Pi_h\vu^*(s)-\vu^*(s)\right>{\rm d}s\bigg].
		\end{align*}
		We use the estimates \eqref{today05} to obtain
		\begin{align*}
			\mathfrak{E}_1&\le\,\|\vB\vx^*\|_{\mathbb{L}^2}\|\vy[\Pi_h\vu^*-\vu^*]\|_{L^2}
			\le C_{T,d}\|\vB\vx^*\|\|\Pi_h\vu^*-\vu^*\|_{\mathbb{L}^2}\\&\le \sqrt{C_{T,d}}\|\vB\vx^*\|_{\mathbb{L}^2} h^{1/2}\sqrt{(\|\vx_0\|^2+\|\vM_1\vx_0\|+\|\sigma\|_{\mathbb{L}^2}^2+ \|\vM_1\sigma\|_{\mathbb{L}^2}^2)}\\
			&\le\,h^{1/2}(\|\vx_0\|^2+\|\vM_1\vx_0\|+\|\sigma\|_{\mathbb{L}^2}^2+ \|\vM_1\sigma\|_{\mathbb{L}^2}^2).
		\end{align*}
		Similarly we can estimate
		\begin{align*}
			\mathfrak{E}_2 +\mathfrak{E}_3\le C_{T,d}h^{1/2}\big(\|\vx_0\|^2+\|\vM_1\vx_0\|+\|\sigma\|_{\mathbb{L}^2}^2+ \|\vM_1\sigma\|_{\mathbb{L}^2}^2\big).
		\end{align*}
		Finally we obtain
		\begin{align}
			\mathfrak{E}\le C_{T,d}h^{1/2}\big(\|\vx_0\|^2+\|\vM_1\vx_0\|+\|\sigma\|_{\mathbb{L}^2}^2+ \|\vM_1\sigma\|_{\mathbb{L}^2}^2\big).
		\end{align}
		\textbf{Step 2}. We estimate $J$ term as follows: By using Young's inequality and \eqref{today11}, for any $\delta>0$ we obtain
		\begin{align*}
			J&\le\,\delta \|(\vu^*-\vu_h^*)\|_{\mathbb{L}^2}^2 + c_\delta \|(\vu^*-\Pi_h\vu^*)\|_{\mathbb{L}^2}^2\\
			&\le\,\delta \|(\vu^*-\vu_h^*)\|_{\mathbb{L}^2}^2 + c_\delta C_{T,d}h(\|\vx_0\|^2+ \|\vM_1\vx_0\|^2 + \|\sigma\|_{\mathbb{L}^2}^2+ \|\vM_1\vsigma\|_{\mathbb{L}^2}^2).
		\end{align*}
		Finally from steps, by choosing small enough $\delta>0$, we conclude that there exist $C_{T,d}>0$ such that
		\begin{align*}
			\|\sqrt{\alpha}(\vu^*-\vu_h^*)\|_{\mathbb{L}^2}^2&+ \|\sqrt{\vB}(\vx^*-\vx_h^*)\|_{\mathbb{L}^2}^2 + \mathbb{E}\big[\|\sqrt{\vD}(\vx^*(T)-\vx_h^*(T))\|^2\big]\\&\le C_{T,d}\,h^{1/2}(\|\vx_0\|^2+ \|\vM_1\vx_0\|^2+\|\vM\vx_0\|^2 + \|\sigma\|_{\mathbb{L}^2}^2+\|\vM\vsigma\|_{\mathbb{L}^2}^2+ \|\vM_1\vsigma\|_{\mathbb{L}^2}^2).
		\end{align*}
	\end{proof}
	\begin{Remark}
		It is important to note that in the proof of the result \eqref{theorem 4.1}, the term \(\mathfrak{E}\) contributes an order of \(h^{1/2}\). In the case of free control, this term becomes identically zero (see Remark \eqref{important remark}), leading to an improvement in the error estimate. 
	\end{Remark}
	We state the following result specifically for the case of free control.
	\begin{Theorem}[Rate of convergence for free control]
		Let $\vx_0\in\mathbb{R}^d$ and $\vU_{ad}=\vU$. Let $\sigma$ satisfies the following estimate:
		\begin{align*}
			\|\sigma-\Pi_h\sigma\|_{\mathbb{L}^2}^2\le\,C_T\,h.
		\end{align*}
		Let $(\vx^*, \vu^*) \in \vX \times \vU$ and $(\vx_{h}^*, \vu_{h}^*)\in \vX_h \times \vU_h$ be the unique optimal pairs for the continuous $\mathbf{SLQ}$ problem and the discrete $\mathbf{SLQ}_h$ problem, respectively. Then the following error bound holds:
		\begin{align}\
			\|(\vu^*-\vu_h^*)\|_{\mathbb{L}^2}^2&+ \|\sqrt{\vB}(\vx^*-\vx_h^*)\|_{\mathbb{L}^2}^2 + \mathbb{E}\big[\|\sqrt{\vD}(\vx^*(T)-\vx_h^*(T))\|^2]\notag\\&\le C_{T,d}\,h(\|\vx_0\|^2+ \|\vM_1\vx_0\|^2+\|\vM\vx_0\|^2 + \|\sigma\|_{\mathbb{L}^2}^2+\|\vM\vsigma\|_{\mathbb{L}^2}^2+ \|\vM_1\vsigma\|_{\mathbb{L}^2}^2).
		\end{align}
	\end{Theorem}
	
	\subsection{Error estimate for projected gradient descent method}
	We want to show the convergence of $\mathbf{SLQ}_{h}^{{\rm grad}}$ ({\em i.e., Algorithm \ref{Algorithm_implementable}}) for $\kappa>0$ sufficiently large as $\ell\to\infty$. For this purpose, we recall the notations $\mathcal{S}_{h}$, $\hat{\mathcal{J}}_{h}$, and $\mathcal{T}_{h}$ introduced in previous sections. First, we recall the Lipschitz continuity of $\mathcal{D}\hat{\mathcal{J}}_{h}$: since
	\[
	\mathcal{D}^2 \hat{\mathcal{J}}_{h}=\mathbf{I}+\alpha\mathcal{S}_{h}^{*}\mathcal{S}_{h},
	\]
	we find $K := \|\mathbf{I} + \mathcal{S}_{h}^{*}\mathcal{S}_{h}\|_{\mathcal{L}(\mathbf{U}_{h}, \mathbf{U}_{h})}$ such that
	\begin{align*}
		\|\mathcal{D}\hat{\mathcal{J}}_{h}(\vu_{h}^1)-\mathcal{D}\hat{\mathcal{J}}_{h}(\vu_{h}^2)\|_{\vU_{h}} \le K\,\|\vu_{h}^1-\vu_{h}^2\|_{\vU_{h}}.
	\end{align*}
	Here, one can find a bound on $K$ with the help of stability estimate \eqref{today 200} such that  
	\begin{align*}
		K &= \|\mathbf{I} + \alpha\mathcal{S}_{h}^{*}\mathcal{S}_{h}\|_{\mathcal{L}({\vU}_{h}, {\vU}_{h})} \le \|\mathbf{I}\| +\alpha \|\mathcal{S}_{h}\|_{\mathcal{L}({\vU}_{h}, \, {\vX}_{h})}^2 \\
		&\le 1 + \alpha \|\vN\|^2T.
	\end{align*}
	Since $\mathbf{SLQ}_{h}^{{\rm grad}}$ is the gradient descent method for $\mathbf{SLQ}_{h}$, we have the following result.
	\begin{Theorem}[Convergence of gradient descent method]\label{convergence of gradient descent method}
		Let $\kappa>K$. Let $\{\vu_{h}^{(\ell)}\}_{\ell\in\mathbb{N}_0}\subset\vU_{h}$ be a sequence of iterates generated by $\mathbf{SLQ}_{h}^{{\rm grad}}$, and $\vu^*_{h}$ be the unique optimal control to the $\mathbf{SLQ}_{h}$ problem. Then
		\begin{align*}
			{\displaystyle\begin{cases}
					\hat{\mathcal{J}}_{h}[\vu_{h}^{(\ell)}]-\hat{\mathcal{J}}_{h}[\vu_{h}^*]\le \frac{2\kappa\|\vu_{h}^{(0)}-\vu_{h}^*\|_{\vU_{h}}^2}{\ell}, \vspace{0.3cm} \\
					\|\vu_{h}^{(\ell)}-\vu^*_{h}\|_{\vU_{h}}^2 \le \bigg(1-\frac{\alpha}{\kappa}\bigg)^\ell \|\vu_{h}^{(0)}-\vu^*_{h}\|_{\vU_{h}}^2, \quad \ell=1,2,3,\ldots
			\end{cases}}
		\end{align*}
		\begin{proof}
			For the proof, we refer to \cite[Theorem 1.2.4]{Ne}.
		\end{proof}
	\end{Theorem}
	
	\begin{Remark}[Finite dimensional approximation of an unbounded operator]\label{Remark 4.0}
		It is worth noting that the constant $C_{T,d}$ depends polynomially on the norms of the given matrices, and is independent from the norm of $\vM$ and $\vM_1$. Due to this, the error analysis encompasses a wide range of stochastic linear quadratic (SLQ) problems related to stochastic differential equations (SDEs), which result from finite-dimensional approximations of $\mathbf{SLQ}$ problems governed by (infinite-dimensional) stochastic partial differential equations (SPDEs). In particular, the matrix $\vM$ can be interpreted as a finite-dimensional approximation of an unbounded operator acting on an infinite-dimensional Sobolev space. For example, $\vM$ may be related to the discrete Laplacian $\Delta_h$ on the finite-dimensional space $V_d(D)$, consisting of piecewise affine  on a finite element mesh contained in $H^1(D)$, where $D$ is a bounded domain. Similarly, $\vM_1$ can be viewed as a matrix associated with the gradient operator. 
	\end{Remark}
	\subsection{Numerical example: rate of convergence}\label{numerical example 01}
	We consider the same data setup as described in Example 1 to ensure consistency in our analysis. Since the explicit form of the true solution is unknown, we adopt an alternative approach by using a reference solution obtained with a very fine mesh size, $h=4\times10^{-3}$. This allows us to estimate the error between the discrete solution and the reference (assumed true) solution effectively. We define
	\begin{align*}
		\mathcal{E}^{(\mathtt{M},L)}(h)=\frac{1}{\mathtt{M}}\sum_{\mathtt{m}=1}^{\mathtt{M}}\|\vx_h^{(\mathtt{m},L)}(T)-\vx_{\text{ref}}^{(\mathtt{m},L)}(T)\|,\qquad\mathcal{E}^{(\mathtt{M},L)}_1(h)=\frac{1}{\mathtt{M}}\sum_{\mathtt{m}=1}^{\mathtt{M}}\|\vx_h^{(\mathtt{m},L)}-\vx^{(\mathtt{m},L)}_{\text{ref}}\|_{L^2([0,T];\mathbb{R}^d)},
	\end{align*}
	\begin{align*}
		\mathcal{E}^{(\mathtt{M},L)}_2(h)=\frac{1}{\mathtt{M}}\sum_{\mathtt{m}=1}^{\mathtt{M}}\|\vu_h^{(\mathtt{m},L)}-\vu^{(\mathtt{m},L)}_{\text{ref}}\|_{L^2([0,T];\mathbb{R}^m)}.
	\end{align*}
	Here $\mathcal{J}_{h}^{(\ell,\mathtt{M})}$ is also defined as follows:
	\[\mathcal{J}_{h}^{(\ell, \mathtt{M})}=\frac{1}{2\mathtt{M}}\sum_{\mathtt{m}=1}^{\mathtt{M}}\left[\int_{0}^{T} \left( \left<\vx_{h}^{(\ell,\mathtt{m})}(s), \vB\vx_{h}^{(\ell,\mathtt{m})}(s)\right>+ \alpha\left< \vu_{h}^{(\ell,\mathtt{m})}(s),\vu_{h}^{(\ell,\mathtt{m})}(s)\right> \right) \,{\rm d}s + \left<\vx_{h}^{(\ell,\mathtt{m})}(T),\vD \vx_{h}^{(\ell,\mathtt{m})}(T)\right>\right], \]
	where $\{(\vx_h^{(\ell,\mathtt{m})},\vu_h^{(\ell,\mathtt{m})})\}_{\mathtt{m}=1}^{\mathtt{M}}$ is the collection of Monte Carlo copies of $(\vx_h^{(\ell)},\vu_{h}^{(\ell)}).$
	\begin{figure}[H]
		\centering
		
		\begin{subfigure}{0.40\textwidth}
			\centering
			\includegraphics[width=\linewidth]{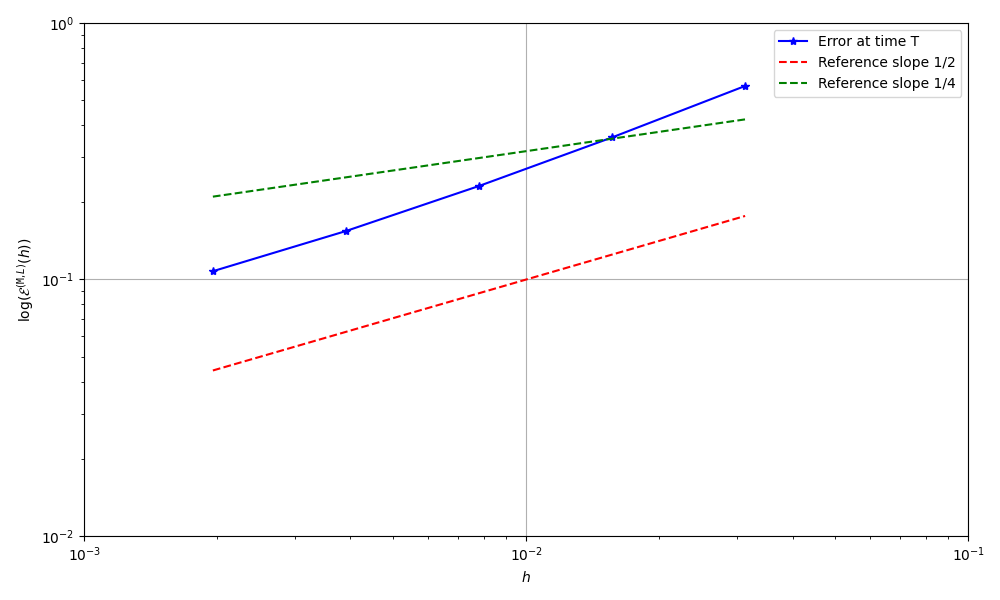}
			\caption{$h\to\log\mathcal{E}^{(\mathtt{M},L)}(h)$}
			\label{fig:image6}
		\end{subfigure}
		\hfill
		\begin{subfigure}{0.40\textwidth}
			\centering
			\includegraphics[width=\linewidth]{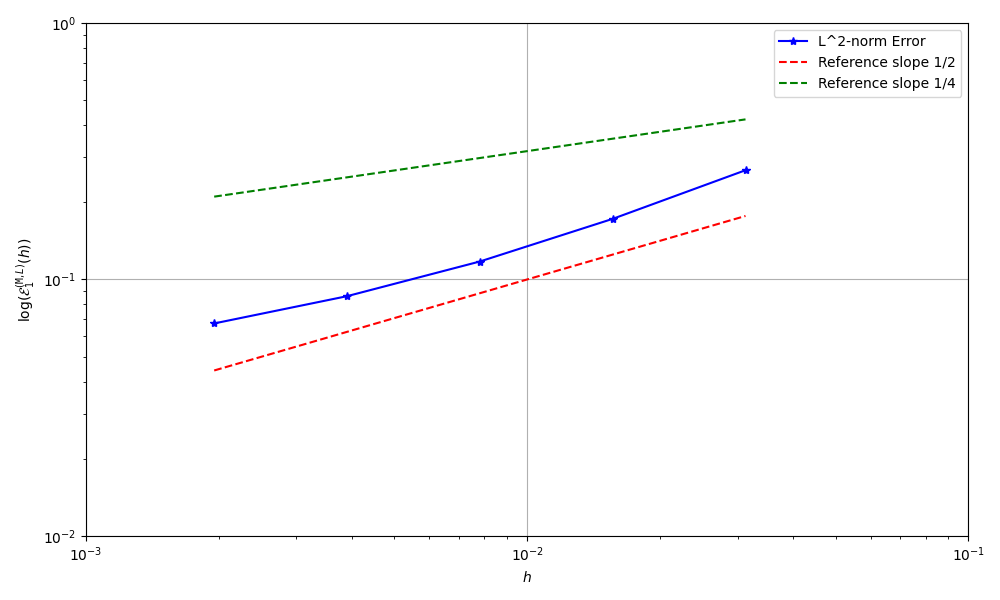}
			\caption{$h\to\log\mathcal{E}_1^{(\mathtt{M},L)}(h)$}
			\label{fig:image8}
		\end{subfigure}
		\hfill
		\begin{subfigure}{0.40\textwidth}
			\centering
			\includegraphics[width=\linewidth]{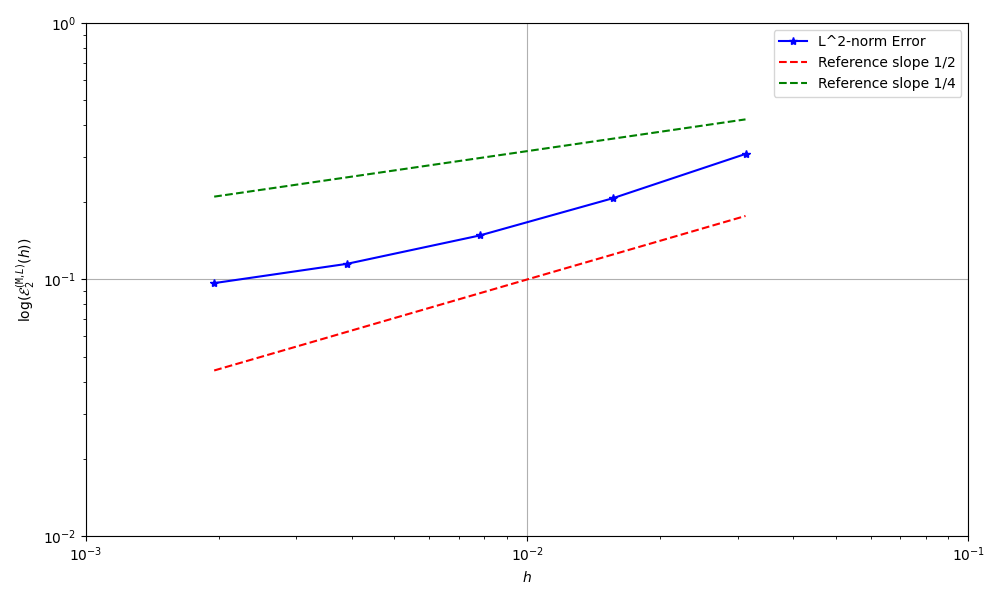}
			\caption{$h\to\log\mathcal{E}_2^{(\mathtt{M},L)}(h)$}
			\label{fig:image10}
		\end{subfigure}
		\caption{\textbf{(a)}  displays the error plot of the state iterate $\vx_h^{(L)}$ at time T, and \textbf{(b)} displays the error plot of state iterate \(\vx_h^{(L)}\) in $L^2$-norm, and \textbf{(c)} displays the error plot of control iterate $\vu_h^{(L)}$ in $L^2$-norm in the case of {\em control constraint}.}
		\label{089}
	\end{figure}
	
	\begin{figure}[H]
		\centering
		
		\begin{subfigure}{0.45\textwidth}
			\centering
			\includegraphics[width=\linewidth]{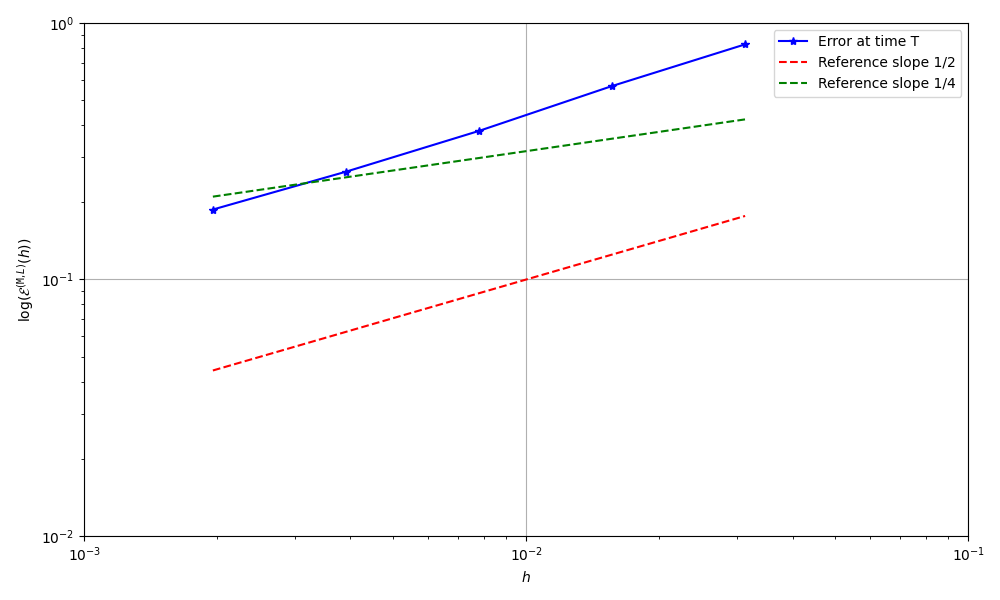}
			\caption{$h\to\log\mathcal{E}^{(\mathtt{M},L)}(h)$}
			\label{fig:image7}
		\end{subfigure}
		\hfill
		\begin{subfigure}{0.45\textwidth}
			\centering
			\includegraphics[width=\linewidth]{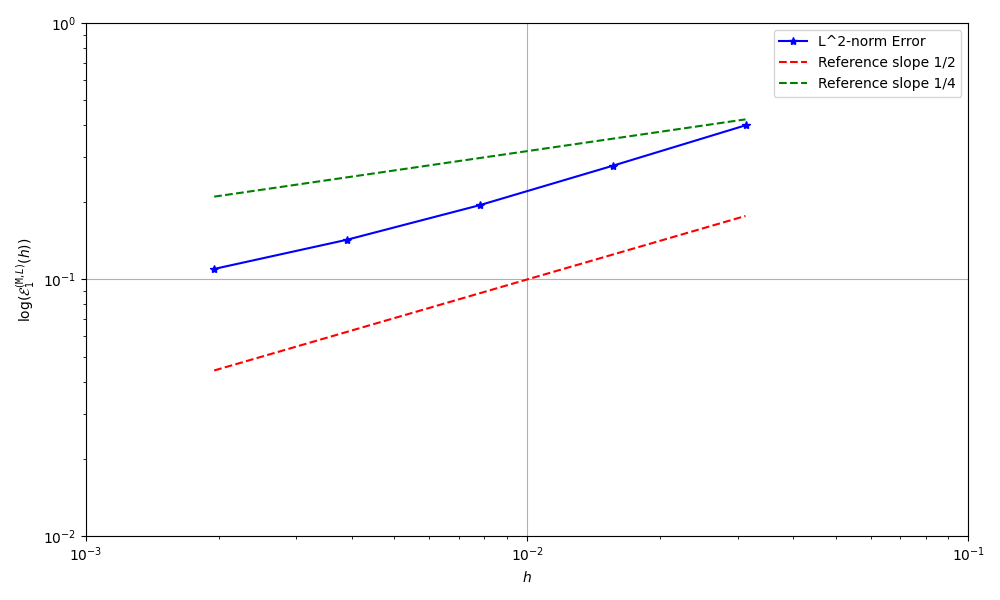}
			\caption{$h\to\log\mathcal{E}_1^{(\mathtt{M},L)}(h)$}
			\label{fig:image9}
		\end{subfigure}
		\hfill
		\begin{subfigure}{0.45\textwidth}
			\centering
			\includegraphics[width=\linewidth]{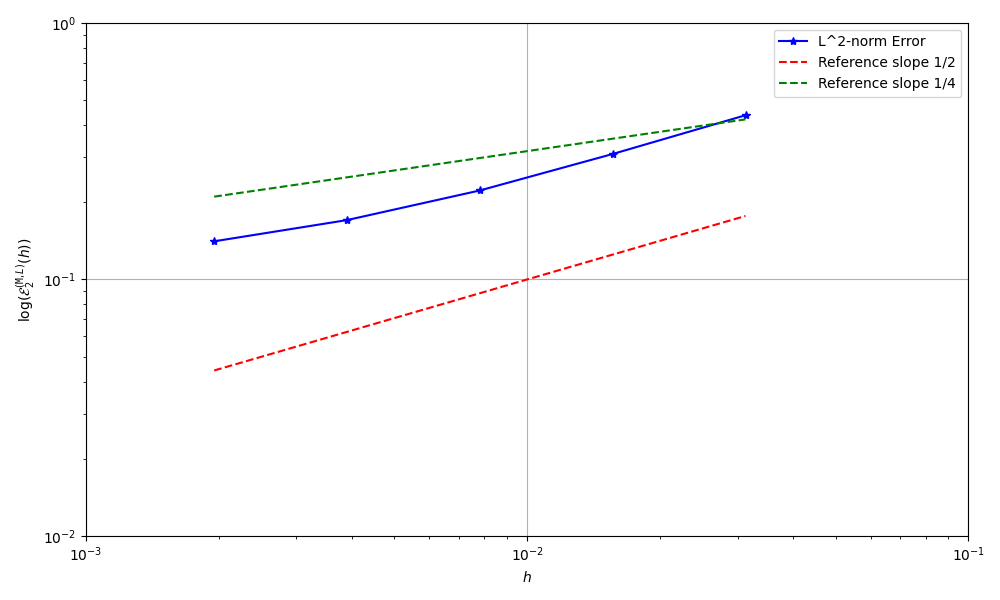}
			\caption{$h\to\log\mathcal{E}_2^{(\mathtt{M},L)}(h)$}
			\label{fig:image11}
		\end{subfigure}
		\caption{\textbf{(a)}  displays the error plot of the state iterate $\vx_h^{(L)}$ at time T, and \textbf{(b)} displays the error plot of state iterate \(\vx_h^{(L)}\) in $L^2$-norm, and \textbf{(c)} displays the error plot of control iterate $\vu_h^{(L)}$ in $L^2$-norm in the case of {\em free control}.}
		\label{090}
	\end{figure}
	\begin{figure}[H]
		\centering
		\begin{subfigure}{0.45\textwidth}
			\centering
			\includegraphics[width=\linewidth]{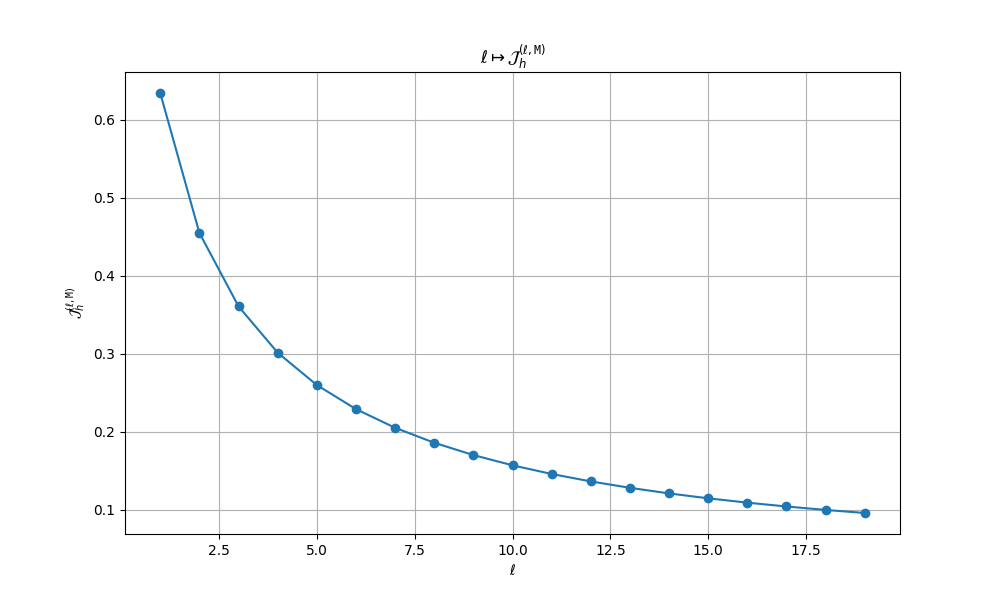}
			\caption{$\ell\to \mathcal{J}_{h}^{(\ell,\mathtt{M})}$}
			\label{fig:image5}
		\end{subfigure}
		
		\caption{\textbf{(a)} illustrates the decay of the cost functional $\mathcal{J}_h^{(\ell,\mathtt{M})}$ with respect to the gradient iterate \(\ell\) in the case of {\em control constraint}.
		}
		\label{fig:all_images}
	\end{figure}
	\begin{Remark}
		Figure \ref{089} illustrates that the rate of convergence is stronger than the order of \(\frac{1}{4}\). The cause behind may be that the stronger convergence is attributed to the regularity of the optimal pair \((\vx^*, \vu^*)\), which is of order \(\frac{1}{2}\), as establish in Corollary \ref{Corollary 2.1}.
		
	\end{Remark}
	\subsection*{Acknowledgments}
	The author wishes to thank Andreas Prohl for many stimulating discussions and valuable suggestions. 

\end{document}